\documentclass[11pt,a5paper]{article}

\IfFileExists{jeb.sty}{\usepackage{jeb}}{\usepackage{ajr}}

\title{Equation-free patch scheme for efficient computational homogenisation via self-adjoint coupling}

\author{J.~E. Bunder
\thanks{School of Mathematical Sciences, University of Adelaide, South Australia~5005, Australia.  
\protect\url{http://orcid.org/0000-0001-5355-2288} ,
\protect\url{mailto:judith.bunder@adelaide.edu.au} } 
\and 
I.~G. Kevrekidis
\thanks{Departments of Chemical and Biomolecular Engineering and Applied Mathematics and Statistics, Johns Hopkins University, Baltimore, Maryland, USA.
\protect\url{http://orcid.org/0000-0002-5313-5887} }
\and 
A.~J. Roberts
\thanks{School of Mathematical Sciences, University of Adelaide, South Australia.  \protect\url{mailto:anthony.roberts@adelaide.edu.au} ,
\protect\url{http://orcid.org/0000-0001-8930-1552} }
}  

\date{\today}

\usepackage{amsmath,amsfonts,amssymb,amsthm,euler,reducecode,doi}
\usepackage[single]{accents}%
\usepackage{defns}
\usepackage{enumitem} 
\usepackage{listings}
\allowdisplaybreaks
\let\LTXappendix\appendix
\def\appendix{{\raggedright\printbibliography}%
    \gdef\printbibliography{}%
    \clearpage\LTXappendix}

\usepackage{pgfplots,tikz} 
\pgfplotsset{compat=newest} 
\usepgfplotslibrary{external} %
\hypersetup{pdftex,
            backref=true,
            hyperindex=true,
            colorlinks=true,
            allcolors=teal!60!black,
            bookmarks=true,
            breaklinks=true}

\usepackage{mybiblatex,mycleveref}

\renewcommand{\mod}{\operatorname{mod}}
\renewcommand{\vec}[1]{\boldsymbol{#1}} 
\def\imath{i}\def\jmath{j}%
\def\diag{\operatorname{diag}}

\Vec u\Vec v\Vec w
\Vec U\Vec V\Vec W
\Vec{kappa}
\Cal C\Cal D\Cal I\Cal J\Cal K\Cal L\Cal N\Cal P\Cal A
\Bb H\Bb R\Bb Z

\def\tv{\RaisedName{\detokenize{\tv}}%
    \tilde v}
\def\tV{\RaisedName{\detokenize{\tV}}%
    \tilde V}
\def\i{\RaisedName{\detokenize{\i}}%
    \operatorname{i}}
\def\Kp{\RaisedName{\detokenize{\Kp}}\hat K}
\def\Km{\RaisedName{\detokenize{\Km}}\check K}
\def\Gamma{{\RaisedName{\detokenize{\Gamma}}P}}

\newcounter{j}\newcounter{i}\newcounter{I}
\def\asinh{\operatorname{asinh}}
\def\C{\mathchar"213B}

\def\ltor{\accentset{\hookrightarrow}}
\def\rtol{\accentset{\hookleftarrow}}

\definecolor{mycolor1}{rgb}{0.00000,0.44700,0.74100}%
\definecolor{mycolor2}{rgb}{0.85000,0.32500,0.09800}%
\definecolor{mycolor3}{rgb}{0.92900,0.69400,0.12500}%
\definecolor{mycolor4}{rgb}{0.49400,0.18400,0.55600}%
\definecolor{mycolor5}{rgb}{0.46600,0.67400,0.18800}%
\definecolor{mycolor6}{rgb}{0.30100,0.74500,0.93300}%
\definecolor{mycolor7}{rgb}{0.63500, 0.07800, 0.18400}%

\newcommand{\vi}[1][i]{\vv^I_{#1}}

\def\RaisedName#1{} %

\begin{document}

\maketitle

\begin{abstract}
Equation-free macroscale modelling is a systematic and rigorous computational methodology for efficiently predicting the dynamics of a microscale system at a desired macroscale system level.
In this scheme, the given microscale model is computed in small patches spread across the space-time domain, with patch coupling conditions bridging the unsimulated space.
For accurate simulations, care must be taken in designing the patch coupling conditions.
Here we construct novel coupling conditions which preserve translational invariance, rotational invariance, and self-adjoint symmetry, thus guaranteeing that conservation laws associated with these symmetries are preserved in the macroscale simulation.
Spectral and algebraic analyses of the proposed scheme in both one and two dimensions reveal mechanisms for further improving the accuracy of the simulations.
Consistency of the patch scheme's macroscale dynamics with the original microscale model is proved. 
This new self-adjoint patch scheme provides an efficient, flexible, and accurate computational homogenisation in a wide range of multiscale scenarios of interest to scientists and engineers.
\end{abstract}

\tableofcontents

\section{Introduction}
\label{sec:intro}

In many complex systems the macroscale dynamics are determined from the coherent behaviour of microscopic agents, such as electrons, molecules, or individuals in a population.
Some complex systems have macroscale models which adequately describe the large-scale dynamics (e.g., diffusion through a homogenous material), but for many others there are no known algebraic closures for a macroscale model and so any accurate description of the system is reliant on resolving microscale structures and interactions which are on a significantly smaller scale than the macroscale of interest.
Furthermore, often a microscale model provides the most accurate description of a system, but its full evaluation is prohibitively expensive for large-scale computations.
Although an approximate macroscale model may be derivable via various multiscale methods, such derivations are often reliant on restrictive assumptions or ad hoc methods which may not be suitable for all systems.
In a recent review for \textsc{nasa}, \citet{NASA2018} discussed the accuracy and adaptability of available multiscale methods and concluded that there is a
``Lack of useful automatic methods for linking models and passing information between scales''.
Our equation-free computational schemes fill this lack by using a given microscale model directly, with no simplification or transformation, and invoking generically crafted coupling conditions to ensure macroscale accuracy.

Equation-free macroscale modelling avoids the derivation of a macroscale model and seeks to overcome computational limitations by computing the given microscale model only within a small fraction of the space-time domain \citep{Kevrekidis04a, Samaey03b, Samaey04}.
As an example, \cref{fig:homoEG2} shows an equation-free simulation of two dimensional diffusion, with microscale heterogeneity in the diffusivities. 
In this simulation the microscale details of the system are not computed in the space between the patches. 
Nonetheless the patch scheme effectively predicts the macroscale dynamics---it is a form of computational homogenisation via a sparse simulation.
In a `Roadmap' prepared for the US Dept of Energy by \cite{Dolbow04}, the scheme proposed here is a Multiresolution, Hybrid, Closure Method.

To suitably classify our equation-free scheme we distinguish the following two terms:
\begin{itemize}
\item \emph{numerical homogenisation} means some numerical computations and analysis of the microscale system that somehow forms a function that serves as a closure for chosen macroscale variables  \citep{Engquist08, Saeb2016};
\item \emph{computational homogenisation} means an on-the-fly purely computational scheme, applied to the multiscale system, that in itself provides an effective closure which reasonably predicts the macroscale dynamics with a computational cost that is essentially independent of the scale separation between micro- and macro-scales~\citep{Kevrekidis09a,Gear03}.
\end{itemize}
The patch scheme developed herein is an example of computational homogenisation.

\begin{figure}
\centering
\caption{\label{fig:homoEG2}four times of a simulation of heterogeneous diffusion~\eqref{eq:full2Ddiff} on small patches in 2D \(xy\)-space coupled by spectral interpolation (\cref{sec:sasc}).  
Here the \(25\)~patches are relatively large for visibility (size ratio \(r=0.4\)).
The heterogeneous, log-normal, diffusivities cause non-trivial sub-patch structures to emerge that are computationally homogenised by our patch scheme.}
\pgfplotsset{label shift=-1.5ex}
\begin{tabular}{@{\hspace{-2mm}}c@{}c@{}}
\includegraphics{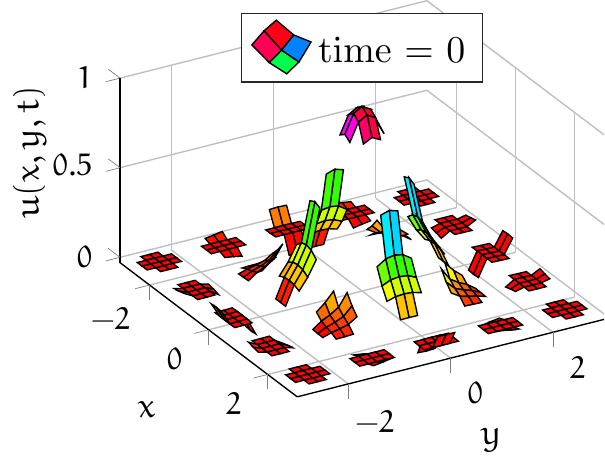}&
\includegraphics{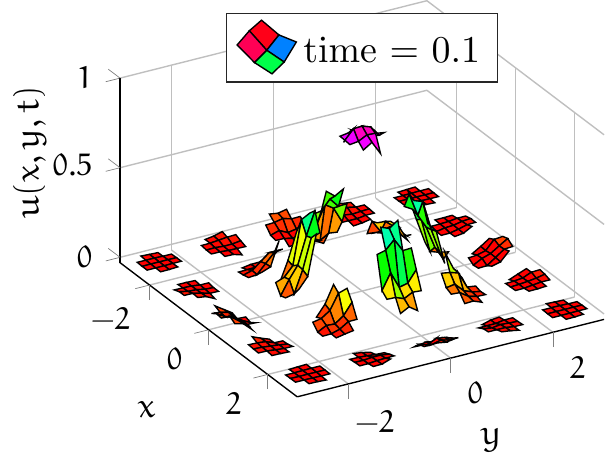}\\
\includegraphics{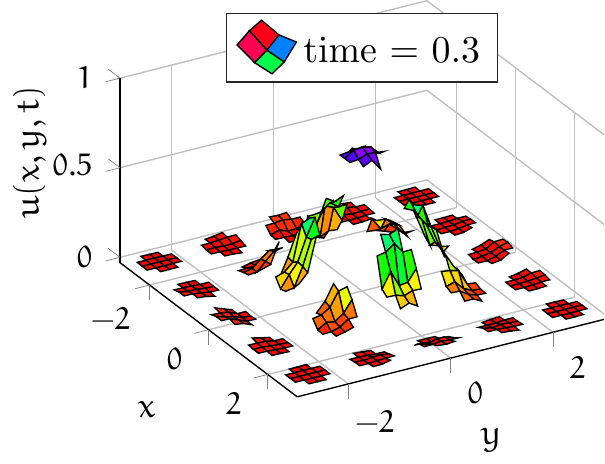}&
\includegraphics{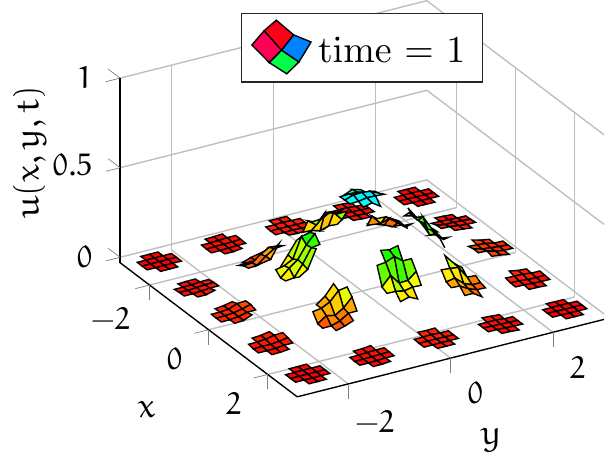}
\end{tabular}
\end{figure}

Much of the initial development of equation-free modelling was concerned with simulations which reduce the size of the temporal domain  and maintain the original spatial domain \citep{Kevrekidis04a, Kevrekidis09a, Samaey03b, Xiu05, Samaey04}.
In this case the numerical solution is constructed from a large number of microscale simulations over short `bursts' of time, with these bursts separated by some macroscale time step (significantly larger than the burst length) and a projective integrator providing the link between successive bursts.
In contrast, \cref{fig:homoEG2} shows a numerical solution for which the spatial domain is reduced to a number of non-overlapping patches, but without projective integration in time implemented. 
In this case, the dynamics which extends across the full spatial domain is captured by patch coupling conditions which interpolate between patches across the unsimulated space.
A full implementation of equation-free modelling combines both projective integration and spatially coupled patches; however, here we focus just on patches in space (often referred to as the gap-tooth method~\citep{Gear03}) because we are concerned with preserving spatial symmetries of the original microscale model.
In particular, we preserve translational invariance, rotational invariance and self-adjoint symmetry by deriving new patch coupling conditions.
\citet{Roberts2010} also constructed self-adjoint coupling conditions, but for overlapping patches, and only in the case of limited coupling.
\cref{sec:pd1D,sec:pd2D} describe our new self-adjoint patch coupling scheme for 1D and~2D space, respectively.

Patch coupling conditions are a crucial component in accurate patch simulations, and for systems with microscale heterogeneity particular care is required. 
\citet{Bunder2017} studied the patch scheme for 1D heterogeneous diffusion and showed accurate coupling conditions should take account of the underlying microscale structure with the interpolation between points strictly controlled by the period of the heterogeneous diffusion.
Although the coupling conditions constructed by \citet{Bunder2017} were shown to be effective for 1D heterogeneous diffusion, as were similar coupling conditions for many other systems~\citep{Roberts13, Cao2015, Jarrad2018}, these coupling conditions fail to preserve self-adjoint symmetry and thus do not maintain a fundamental symmetry of the original problem.
A system represented by $\partial_t\uv=\cL\uv$ with field~$\uv$ and linear operator~\cL\ is defined to be \emph{self-adjoint} if for all fields~$\vv,\uv$, \cL~satisfies $\langle \uv,\cL\vv\rangle=\langle \cL\uv,\vv\rangle$ for some inner product.
Here, we consider lattice systems with square matrix operator~\cL\ and apply the usual complex inner product $\langle \uv,\vv\rangle=\uv^{\dag}\vv$\,, where $\dag$~denotes the complex conjugate transpose.
To ensure that operator~\cL\ is self-adjoint it must satisfy $\cL^{\dag}=\cL$, termed \emph{Hermitian}.\footnote{Often~\cL\ is a real matrix, and then this Hermitian property is the usual matrix symmetry.}
An approximation scheme which does not maintain self-adjoint symmetry may result in a solution which does not satisfy essential requirements, such as conservation of energy. 
In particular, generalising the 1D patch scheme for heterogeneous diffusion developed by \citet{Bunder2017} to~2D heterogeneous diffusion produces undesirable fluctuations in the simulation.
These fluctuations arise because such a 2D~patch system produces a non-self-adjoint Jacobian that possessed complex eigenvalues.
Here we construct new patch coupling conditions which maintain the self-adjoint symmetry of the original microscale system, for both 1D and~2D space.
By preserving self-adjoint symmetry, these new coupling conditions have \text{much wider applicability.}

The code developed herein now forms part of a flexible \textsc{Matlab}\slash Octave Toolbox \cite[]{Maclean2020a} for equation-free computations that any researcher can download and use for a variety of problems \cite[]{Roberts2019b}.   
This Toolbox provides equation-free code suitable for many systems, such as diffusion and wave dynamics, and grants the user full flexibility in selecting one of the supplied projective integration and\slash or patch coupling schemes or importing user-written code.

To clarify the theoretical results and demonstrate the benefits of our novel self-adjoint patch scheme we use the examples of 1D and~2D microscale heterogeneous diffusion (\cref{sec:pd1D,sec:theory}, and \cref{sec:pd2D}, respectively).
For the 2D~case we pose that the system's evolving variables~\(u_{i,j}(t)\) are defined on a spatial lattice of points~\((x_i,y_j)\) with microscale spacing~\(d\) (\(\Delta x_i=\Delta y_j=d\)) and indexed by integers~\(i,j\). 
The governing large set of \ode{}s of the heterogeneous diffusion is then   
\begin{align}
d^2\partial_t u_{i,j}&=\kappa_{i+\frac12,j}(u_{i+1,j}-u_{i,j})+\kappa_{i-\frac12,j}(u_{i-1,j}-u_{i,j})\nonumber\\
&\quad{}+\kappa_{i,j+\frac12}(u_{i,j+1}-u_{i,j})+\kappa_{i,j-\frac12}(u_{i,j-1}-u_{i,j})
\label{eq:full2Ddiff}
\end{align}
for heterogeneous diffusivities~$\kappa_{i,j}$ which vary periodically over the given domain.
\cref{fig:homoEG2} illustrates such a lattice, albeit restricted to patches in space rather than space-time.
Here, for clarity of notation, and as shown in the figure, we assume a square domain and a square microscale lattice, but \cref{sec:pd2D} generalises to a rectangular domain and rectangular microscale lattice. 
We use the relatively simple example of heterogeneous diffusion as it is a canonical example which describes many physical systems and naturally extends to more complex systems, such as wave propagation and advection-diffusion, as briefly discussed in \cref{sec:fwsc}.

There are many multiscale methods which would provide good solutions to multiscale heterogeneous systems such as~\eqref{eq:full2Ddiff}, as reviewed by both \cite{Engquist08, Saeb2016}. 
For example, \cite{Abdulle2011, Engquist2011} applied the heterogeneous multi-scale method (\textsc{hmm}) to wave propagation in heterogeneous media, although it is predicated on an infinitely large scale separation between the `slow' variables which persist at the macroscale and the `fast' variables which are only observable at the microscale. 
\cite{Maier2019, Peterseim2019} considered a similar wave propagation model, but avoided the need for infinite  scale separations by applying localized orthogonal decomposition with numerical homogenization.
\cite{Romanazzi2016} also applied \textsc{hmm} to compute a macroscale model for a heterogeneous system of closely packed insulated conductors. 
\cite{Owhadi2015} investigated heterogeneous diffusion with numerical homogenization, but reformulated it as a Bayesian inference problem, thus providing a new methodology for deriving basis elements of the microscale structure. 
\cite{Cornaggia2020} considered  one-dimensional waves in periodic media and derived homogenized boundary and transmission conditions in the usual separation of scales infinite limit.
Homogenization and \textsc{hmm} multiscale models rely on being able to identify `fast' and `slow' variables in the microscale model and some prior knowledge of the macroscale model, and generally these models require substantial analytic work prior to numerical implementation~\citep{Carr2016} while also relying on an infinite separation of scales.
In contrast, equation-free modelling makes no assumptions concerning fast and slow variables, needs no limit on the separation of scales, and requires no knowledge of the macroscale model, but instead computes a numerical macroscale solution on-the-fly.
Consequently, equation-free modelling provides macroscale system-level predictions for complex dynamical systems which cannot be solved via other schemes.

\section{Self-adjoint preserving patch scheme for 1D}
\label{sec:pd1D}

The general scenario is when scientists or engineers has a well-specified microscale system on a characteristic microscale length~$d$, but are only interested in the behaviour of this system at some significantly larger macroscale $H\gg d$\,.
An important class of examples is the prediction of the macroscale, homogenised, dynamics of the microscale heterogeneous diffusion of a field~\(u(x,t)\) satisfying the \pde\ \(\partial_tu=\partial_x[\kappa(x)\partial_xu]\) where the diffusivity varies rapidly on the length-scale~\(d\) \citep[e.g.]{Engquist08, Saeb2016, Bunder2017}.
A distinguishing feature of our patch dynamics approach is that we do \emph{not} require infinite scale separation ``\(d\to0\)''; instead the methodology applies at finite~\(d\) via supporting theory which is directly applicable to finite~\(d\)~\cite[c.f.,][who apply scale separation $\epsilon \rightarrow 0$]{Engquist08}.

For simplicity, we introduce our approach for systems defined on a microscale spatial lattice, such as a spatial discretisations of \pde{}s on the microscale length~\(d\).
This section introduces how to construct patches for a 1D system so that a computational simulation on these patches efficiently and accurately predicts the macroscale of interest without needing to derive a macroscale closure. 
In this patch construction we focus on the specific example of 1D~heterogeneous diffusion on a microscale lattice, but the new patch scheme is similar for a wide range of 1D~systems, including nonlinear systems \cite[e.g., via the toolbox by][]{Maclean2020a}.

\begin{figure}
\centering
\caption{three 1D spatial patches indexed by $I$~and~$I\pm 1$\,. 
Filled circles indicate patch interior points with microscale spacing~$d$, and unfilled circles represent patch edge points. 
The gaps between patches are unsimulated space.
These patches have $n=5$ interior points indexed by $i=1:n$\,, and the two edges have indices $i=0$ and $i=n+1=6$\,.
The macroscale spacing between patches is~$H$, and the width of each patch is denoted $h:=nd$, as shown. }
\label{fig:1Dpatch}
\includegraphics{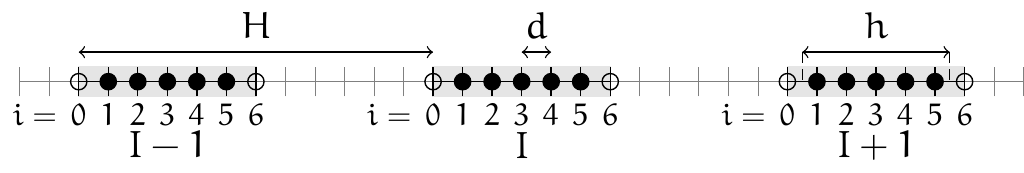}
\end{figure}

In~1D, the lattice has points~\(x_i\) of microscale spacing~\(d\) and we seek to predict the dynamics of the variables~\(u_i(t)\).
Heterogeneous diffusion on this 1D microscale lattice is the restriction of~\eqref{eq:full2Ddiff}, namely
\begin{equation}
d^2\partial_t u_{i}=\kappa_{i+\frac12}(u_{i+1}-u_{i})+\kappa_{i-\frac12}(u_{i-1}-u_{i}).
\label{eq:full1Ddiff}
\end{equation}
As illustrated in \cref{fig:1Dpatch}, to implement the 1D patch scheme we construct \(N\)~small patches across the spatial domain, separated by macroscale length~$H$ and indexed by $I=1,2,
\ldots, N$\,.
Generally we use uppercase letters to denote macroscale quantities, and lowercase letter to denote microscale quantities.
Each patch has $n$~\emph{interior} microscale points indexed by $i=1: n$ (herein, \(k:\ell\) denotes \(k,k+1,k+2,\ldots,\ell\)) and \emph{two edge} points indexed by $i=0$ on the left and $i=n+1$ on the right: in the \(I\)th~patch the location of these points is denoted by~\(x^I_i\).  
We call~\(n\) the \emph{size} of the patches, as opposed to the physical patch \emph{width}~$h$.
We now relabel the microscale field values, using~$u^I_{i}$ to denote the value of~\(u\) in patch~$I$ at microscale patch interior index~$i$.
Similarly, $\kappa_{i+1/2}^I$ is the diffusivity between the $i$~and~$i+1$ points in the $I$th~patch.  
The patch scheme uses the given microscale system~\eqref{eq:full1Ddiff} inside each patch (as a given `black-box'), here the sub-patch \ode{}s are
\begin{equation}
d^2\partial_t u^I_i=\kappa^I_{i+\frac12}(u^I_{i+1}-u^I_{i})
+\kappa^I_{i-\frac12}(u^I_{i-1}-u^I_{i})\,,
\label{eq:1Ddiff}
\end{equation}
on the interior points $i=1:n$ for every patch~$I$, where `sub-patch' refers to spatial scales less than the patch width~$h$.
These patches are coupled together by setting the patch-edge values~$u^I_0$ and~\(u^I_{n+1}\) through interpolation of \(u\)-values from neighbouring patches (\cref{fig:1Dcoupling}).
The new scheme is to interpolate from the interior values closest to the patch edges, the next-to-edge values~$u_1^I$ and~\(u^I_n\), to determine the edge values~\(u^I_{n+1}\) and~\(u^I_0\), respectively.
As illustrated in \cref{fig:1Dcoupling}, the interpolation is from the \emph{opposite} patch edge so that  \(u\)-values at $i=1$ in nearby patches~$I$ are interpolated to the edge \(u\)-value at $i=n+1$, and \(u\)-values at $i=n$ in nearby patches~$I$ are interpolated to the edge \(u\)-value at $i=0$\,.

This article establishes that this new patch scheme both preserves self-adjoint symmetry and also accurately captures the macroscale behaviour of the full heterogeneous system~\eqref{eq:full1Ddiff}.

This self-adjoint coupling is analogous to that applied by \cite{Gear03} in a gap-tooth method particle simulation.
In their particle simulation, at the right edge of patch~$I$, some fraction~$\alpha\leq 1$ of the right-moving particles (i.e., those particles leaving patch~$I$ via the right edge) enter the left edge of adjacent patch~$I+1$, and the remaining fraction~$1-\alpha$ enter the left edge of patch~$I$.
Similarly, left-moving particles at the left edge of patch~$I$ (i.e., those particles leaving patch~$I$ via the left edge) are either distributed to the right edge of patch~$I-1$ or to the right edge of patch~$I$. 
\cite{Gear03} justified their coupling using a flux analogy.

The patch width~$h:=nd$ is measured across the interior points from/to midway between the extreme pairs of microscale points (\cref{fig:1Dpatch}).
The crucial ratio $r:=h/H$ is, in~1D, the fraction of the given spatial domain on which microscale computation takes place.
For efficient simulations we typically choose ratio $r\ll 1$\,, equivalently \(h\ll H\)\,. 
Herein, examples illustrating system-level predictions often use $r=0.1$ so that there is a significant proportion of uncomputed space---the patch scheme is a sparse simulation; however, in example simulations (e.g., \cref{fig:homoEG2,fig:homoExample1Uxt}), for better visualisation we often choose larger~$r$.
When the ratio $r=1$ the patch scheme computes the given microscale system, here~\eqref{eq:full1Ddiff}, since the left edge $i=0$ of the $I$th~patch coincides with the $i=n$ next-to-edge point of the $(I-1)$th~patch, and the right edge $i=n+1$ of the $I$th~patch coincides with the $i=1$ next-to-edge point of the $(I+1)$th~patch, and thus the patches cover all microscale points in the domain.
\cref{sec:sasc,sec:fwsc,sec:ens} demonstrate here, and \cref{sec:theory} proves in general, that patches with small fraction~\(r\) accurately predict \text{the macroscale dynamics.}

\begin{figure}
\centering
\caption{\label{fig:1Dcoupling}
patch self-adjoint preserving coupling.
Filled {\color{blue} triangles}\slash {\color{red} circles} at indices $i={\color{blue} n},{\color{red} 1}$ represent next-to-edge points whose \(u\)-values are interpolated to give \(u\)-values at the edge points, $i={\color{blue} 0},{\color{red} n+1}$, represented by unfilled {\color{blue} triangles}\slash {\color{red} circles}.
Interpolation to the {\color{blue}left}\slash {\color{red}right} edge of a patch is from the next-to-edge points on the {\color{blue}right}\slash {\color{red}left} of nearby patches, as indicated by {\color{blue} blue}\slash {\color{red}red} arrows.}
\includegraphics{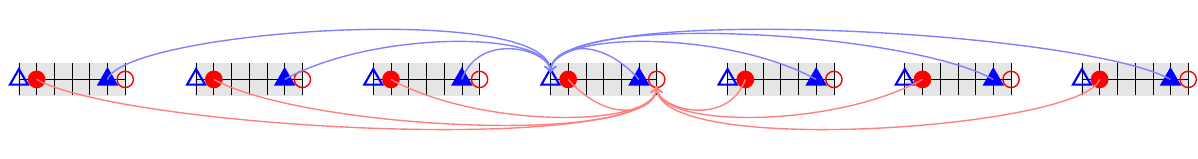}
\end{figure}

\subsection{Self-adjoint patch coupling for 1D}
\label{sec:sa1D}

A key requirement in accurate patch simulations is carefully chosen patch coupling. 
Most previous patch implementations~\citep[e.g.,][]{Roberts13, Cao2015, Bunder2017} interpolated the \(u\)-values from the centre of each patch 
to determine the \(u\)-values on the patch edges.
This was proven to be accurate for smooth dynamics in both 1D and~2D, as well as for 1D heterogeneous diffusion \citep{Bunder2017}.
However, this centre interpolated coupling scheme has one weakness: it does not preserve the self-adjoint nature, the symmetry, of many physical systems, such as the diffusion~\eqref{eq:full1Ddiff}.
This section establishes, via \cref{lemsasp,lemsafw}, that our new patch coupling (\cref{fig:1Dcoupling}) preserves the self-adjoint symmetry of the \text{original microscale system.}

The dynamic variables in the patch scheme are the \(u\)-values at the interior points of the patches.
Hence, for every patch~\(I\), denote the vector of interior values by \(\uv^I:=(u^I_1,\ldots,u^I_n)\in\RR^n\).
These vectors do not contain the patch edge values~$u^I_0$ and~$u^I_{n+1}$ because edge values are determined by interpolation from nearby patches.
Denote the global \(u\)-values by the vector \(\uv:=(\uv^1,\ldots,\uv^N)\in\RR^{nN}\).
Then the microscale linear diffusion~\eqref{eq:1Ddiff} on the coupled patches has the form of the linear systems of \ode{}s
\begin{equation}
\partial_t\uv =\cL \uv \,,\label{eq:diffnew}
\end{equation}
for real-valued $nN\times nN$  matrix~\cL. 
From the patch structure, the square matrix~\cL\ has the form of \(N\times N\) blocks, each of microscale size~\(n\times n\).
Let $\cL^{IJ}$ be  the \(n\times n\) block of the influence on patch~\(I\) from patch~\(J\). 
Here, for heterogeneous systems such as~\eqref{eq:1Ddiff}, we establish that the inter-patch coupling shown by \cref{fig:1Dcoupling} ensures that \cL~is self-adjoint, which in turn ensures accurate macroscale predictions. 
As discussed in \cref{sec:intro}, \cL~is self-adjoint when it is equal to its Hermitian conjugate, $\cL=\cL^{\dag}$\,, and so when \cL~is real we only require it to be symmetric. 

One crucial proviso for these results is that the patch width~\(h=nd\) must be an integral multiple of the period of the microscale heterogeneity. 
The desirability of such a limitation on the patch width has been observed before \cite[\S5.2]{Bunder2013b}
\cite[p.3, Ref.~5,19,20, e.g.]{Abdulle2020b} 
\cite[p.62]{Abdulle2012}.
However, \cref{sec:ens} shows that, and \cref{sssech} proves that, embedding the given heterogeneous system into an ensemble overcomes this limitation on \text{the patch width.}

We decompose the matrix in~\eqref{eq:diffnew} into two parts, $\cL=\cD+\cC$, for block diagonal dynamics matrix $\cD=[\cD^{I}]$ and coupling matrix $\cC=[\cC^{IJ}]$ with $n\times n$~blocks $\cD^{I}$~and~$\cC^{IJ}$.
For~$\cL$ to be self-adjoint requires that both \cD~and~\cC\ are self-adjoint, that is, $(\cD^{I})^{\dag}=\cD^{I}$ and $(\cC^{IJ})^{\dag}=\cC^{JI}$.
The blocks of the dynamics matrix~$\cD^I$ encode the given sub-patch \ode{}s~\eqref{eq:1Ddiff} within the $I$th~patch, and without any patch coupling.
Thus \cD~is self-adjoint if the given microscale system is self-adjoint.
For example, for 1D microscale diffusion \ode{}s~\cref{eq:1Ddiff} the elements of the dynamics matrix are, for every patch~\(I\),
\begin{align}
\cD^I_{i,i}&=-\kappa^I_{i+\frac12}-\kappa^I_{i-\frac12}\quad \text{for } i=1:n\,,\nonumber\\
\cD^I_{i,i+1}&=\cD^I_{i+1,i}=\kappa^I_{i+\frac12}\quad \text{for } i=1:n-1\,,\label{eq:D}
\end{align}
and therefore \cD\ is symmetric and real, and thus self-adjoint.

We now construct a self-adjoint coupling matrix~\cC\ specifically for 1D heterogeneous diffusion~\cref{eq:1Ddiff}, but readily adaptable to any 1D system of second order in space.
We make the following two assumptions that reflect that we want the sub-patch system to be almost a `black-box'. 
\begin{enumerate}[label=\textit{A\arabic*},ref=Assumption~A\arabic*]
\item 
\label{A1} The microscale \ode{}s~\eqref{eq:1Ddiff} are unmodified for all $i$ and~\(I\).
\item 
\label{A2} The patch edge values~$u^I_0$ and~$u^I_{n+1}$ are determined by an interpolation that is independent of the diffusion coefficients. 
\end{enumerate}

\cref{A1} requires that rows \(i=2:n-1\) of~\(\cC^{IJ}\)\ must be zero because~\(\cD^I\)\ already encodes the \ode{}s~\eqref{eq:1Ddiff} for $i=2:n-1$\,.
Consequently, since symmetry requires $\cC^{IJ}=(\cC^{JI})^{\dag}$\,, the columns \(j=2:n-1\) must be zero.
Thus the only nonzero elements in $\cC^{IJ}$ are the four corner elements~$\cC^{IJ}_{ij}$ for $i,j=1,n$\,.
Given the form of~\cD\ in~\eqref{eq:D}, and the form of the \ode{}s~\eqref{eq:1Ddiff}, to satisfy \cref{A1} for the two cases $i=1,n$ we must have
\begin{equation}
\kappa^I_{\frac12}u^I_0=\sum_J\big(\cC^{IJ}_{11}u^J_1+\cC^{IJ}_{1n}u^J_n\big)
\quad \text{and} \quad 
\kappa^I_{n+\frac12}u^I_{n+1}=\sum_J\big(\cC^{IJ}_{n1}u^J_1+\cC^{IJ}_{nn}u^I_n\big).\label{eq:interp}
\end{equation}
These two equations couple patches across the unsimulated space between patches by setting the two patch edge values $u^I_0$~and~$u^I_{n+1}$ as interpolations of the sub-patch fields $u^J_1$~and~$u^J_n$.

To satisfy \cref{A2} we introduce an $nN\times nN$~\emph{interpolation} matrix~$\cI$ that has no dependence on the diffusion coefficients, with the same block structure as~\cC\ where $\cI^{IJ}_{ij}:=0$ for $i\text{ or }j\neq 1,n$\,, whereas   for $j=1,n$ the entries satisfy  $\cC^{IJ}_{1j}=\kappa^I_{1/2} \cI^{IJ}_{1j}$ and $\cC^{IJ}_{nj}=\kappa^I_{n+1/2} \cI^{IJ}_{nj}$\,.
Consequently, \eqref{eq:interp}~simplifies so that the interpolation to the patch-edge values is independent of the diffusion coefficients:
\begin{equation*}
u^I_0=\sum_J\big(\cI^{IJ}_{11}u^J_{1}+\cI^{IJ}_{1n}u^J_n\big) 
\quad \text{and} \quad
u^I_{n+1}=\sum_J\big(\cI^{IJ}_{n1}u^J_1+\cI^{IJ}_{nn}u^J_n\big).
\end{equation*}
For the interpolation to generically interpolate real-values to real-values we thus require matrix~\(\cI\), and hence~$\cC$, to be real.
But we have not yet completely ensured $\cC$~is self-adjoint.

For $\cC$~to be self-adjoint we require that $\cC^{IJ}_{1n}=\kappa^I_{1/2}\cI^{IJ}_{n1}=\cC^{JI}_{n1}=\kappa^J_{n+1/2}\cI^{JI}_{1n}$\,, $\cC^{IJ}_{11}=\kappa^I_{1/2} \cI^{IJ}_{11}=\cC^{JI}_{11}=\kappa^J_{1/2} \cI^{JI}_{11}$ and $\cC^{IJ}_{nn}=\kappa^I_{n+1/2} \cI^{IJ}_{nn}=\cC^{JI}_{nn}=\kappa^J_{n+1/2} \cI^{JI}_{nn}$\,, and since the interpolation coefficients are independent of the diffusion they must satisfy $\cI^{IJ}_{1n}={\cI}^{JI}_{n1}$\,, $\cI^{IJ}_{11}={\cI}^{JI}_{11}$ and $\cI^{IJ}_{nn}={\cI}^{JI}_{nn}$\,.
But if the interpolation matrix `top-left' and `bottom-right' elements~$\cI^{IJ}_{11}$ and~$\cI^{IJ}_{nn}$ are nonzero, then they cannot produce an accurate interpolation. 
To see why, say $J\geq I$ so that~$u^J_1$ is a distance of~$H(J-I)+d$ from~$u^I_0$, and $u^I_1$~is a distance of~$H(J-I)-d$ from~$u^J_0$---these different distances imply that if coefficients $\cI^{IJ}_{11}$~and~$\cI^{JI}_{11}$ are nonzero, then they should not have the same magnitude.
A similar argument implies that nonzero $\cI^{IJ}_{nn}$~and~$\cI^{JI}_{nn}$  should not have the same magnitude.
Thus we set $\cI^{IJ}_{11}=\cI^{IJ}_{nn}=0$\,.
Consequently, the only coupling entries which may be nonzero are $\kappa^I_{1/2}\cI^{IJ}_{n1}=\kappa^J_{n+1/2}\cI^{JI}_{1n}$\,, and for $\cI^{IJ}_{n1}=\cI^{JI}_{1n}$ our diffusion is then constrained by $\kappa^I_{1/2}=\kappa^J_{n+1/2}$ for all patches $I,J$.
The resulting interpolation gives edge values 
\begin{equation}
u^I_0=\sum_J\cI^{IJ}_{1n}u^J_n
\quad \text{and} \quad 
u^I_{n+1}=\sum_J\cI^{IJ}_{n1}u^J_1\,.\label{eq:sacoup}
\end{equation}
\cref{fig:1Dcoupling} draws this interpolation and shows that for symmetric interpolation matrix~\(\cI\) the interpolation is implemented with both translational and rotational symmetry  (i.e., interpolations are invariant upon reflection and swapping red-blue).

Thus \cref{A1,A2} not only tightly constrain the form of coupling matrix~$\cC$, they also constrain the patch size and placement by requiring that the diffusivities satisfy $\kappa^I_{1/2}=\kappa^J_{n+1/2}$ for every patch~$I,J$.
If the heterogeneity is periodic, period~\(p\), then this constrains the patch width~$h$ by requiring the number of patch interior indices~$n$ to be divisible by the heterogeneous period~$p$.
\cref{sec:ens} shows that this constraint on the patch width may be overcome by embedding the diffusion into an ensemble of phase-shifted diffusions, but as this embedding increases the size of the simulation by a factor of~$p$, it incurs extra computation and so might not be suitable for all applications.

The self-adjoint matrix operator~\cL\ is here developed in the context of 1D heterogeneous diffusion.
Nonetheless, the same coupling matrix~\cC\ could be used for any system where the dynamics matrix~\cD\ is self-adjoint and the given microscale model is no more than second order in space, and thus requiring only one coupling condition on each patch edge. 
Higher order microscale models, for example, microscale models fourth order in space, require two coupling conditions for each patch edge and therefore require a more complicated coupling matrix~\cC.

\cref{sec:sasc} discusses the case of global spectral coupling, whereas \cref{sec:fwsc} discusses coupling from a finite number of near neighbouring patches. 
Such patch couplings have different coupling matrices~$\cC$, but for both $\cL=\cD+\cC$ is self-adjoint, and both satisfy \cref{A1,A2}.

\subsection{Spectral coupling of patches}
\label{sec:sasc}

This coupling uses a spectral interpolation of selected patch interior values to give the values on patch edges.
Here we assume the macroscale solution is \(L\)-periodic in space, for domain length \(L:=NH\)\,.
This spectral coupling is very accurate, as indicated by consistency arguments in \cref{sec:theory} and by numerical tests in \cref{sec:sa2D} for the 2D case which also hold in 1D.

Let's start by determining the patches' right-edge values~$u^I_{n+1}$.
The first step is to compute the Fourier transform, the \(N\)~coefficients~$\ltor{u}_{k}$, of the left-next-to-edge values~\(u^I_1\) (the over-harpoon to the right denotes using left-values to determine right-values).  
That is, recalling $x^I_1$~is the spatial location of~$u^I_1$, the Fourier transform computes the coefficients (\(\i:=\sqrt{-1}\))
\begin{subequations}\label{eqs:1Dspec}%
\begin{align}&
\ltor{u}_{k}:=\frac1N\sum_{I}u^I_1e^{-\i kx^I_1},
&\text{so that }&
u^I_1=\sum_{k}\ltor{u}_{k}e^{\i kx^I_1},\quad I=1:N\,.
\label{eq:1Dspec1}
\end{align} 
The wavenumbers~$k$ in these Fourier transforms are an appropriate fixed set of integer multiples of~$2\pi/L$ for domain length~\(L\).
The second step computes the right-edge values by evaluating this Fourier transform at the right-edge of each patch, namely at~\(x^I_{n+1}=x^I_1+h\) for the displacement of one patch width $h=nd$\,. 
Hence
\begin{equation}
u^I_{n+1}  
:=\sum_{k}\ltor{u}_ke^{\i kx^I_{n+1}}
=\sum_{k}\ltor{u}_ke^{\i k(x^I_1+h)} 
=\sum_{k}\big(\ltor{u}_ke^{\i kh}\big)e^{\i kx^I_1},
\label{eq:1Dspecn1}
\end{equation}
which is efficiently realised by computing the inverse Fourier transform of the values~\(\ltor{u}_ke^{\i kh}\) for the range of wavenumbers~\(k\).

Similarly, to compute left-edge values~$u^I_0$ a Fourier transform computes the coefficients~\(\rtol{u}_{k}\) of the right-next-to-edge values~$u^I_n$ using the same set of wavenumbers, shifts a displacement~\(-h\) by multiplying by~\(e^{-\i kh}\), and then an inverse Fourier transform interpolates to the patch left-edge: from the Fourier transform
\begin{align}
u^I_n&=\sum_{k}\rtol{u}_{k}e^{\i kx^I_n}, 
&\text{then }
u^I_0&:=\sum_{k}\big(\rtol{u}_ke^{-\i kh}\big)e^{\i kx^I_n}.\label{eq:1Dspec0}
\end{align}
\end{subequations}

\begin{lemma}\label{lemsasp}
Coupling patches with the spectral interpolation~\eqref{eqs:1Dspec} preserves the self-adjoint symmetry of heterogeneous systems in the form~\eqref{eq:1Ddiff}.
\end{lemma}

Some algebra now establishes this lemma.
The effect of the spectral interpolation to the right-edge is, recalling the patch-width \(h=rH\)\,,
\begin{subequations}\label{eqs:ccsp}%
\begin{align}
u^I_{n+1}&
:=\sum_{k}\ltor{u}_ke^{\i kh}\,e^{\i kx^I_1}
=\sum_{k}\frac1N\sum_{J}u^J_1e^{-\i kx^J_1} e^{\i kh}\,e^{\i kx^I_1}
\nonumber\\&
=\sum_{J}u^J_1\sum_{k}\frac1Ne^{\i k(-x^J_1+h+x^I_1)}
=\sum_{J}u^J_1\sum_{k}\frac1Ne^{\i kH(I-J+r)}
\nonumber\\&
=\sum_{J}\cI^{IJ}_{n1}u^J_1
\qquad\text{for }\cI^{IJ}_{n1}:=\frac1N\sum_{k}e^{\i kH(I-J+r)}.
\end{align}
Similarly, the effect of the spectral interpolation to the left-edge is
\begin{align}
u^I_0&=\cdots
=\sum_{J}\cI^{IJ}_{1n} u^J_n
\qquad\text{for }\cI^{IJ}_{1n}:=\frac1N\sum_{k}e^{\i kH(I-J-r)}.
\end{align}
\end{subequations}
These expressions determine the coefficients in the coupling matrix~\cC.

Since the sums in~\cref{eqs:ccsp} all use the same set of wavenumbers~\(k\), then every pair of interpolation matrix elements \(\cI^{IJ}_{1n}\) and~\(\cI^{JI}_{n1}\) are the complex conjugate of each other.
In real problems, provided every wavenumber~\(k\) in these Fourier sums is partnered by the corresponding negative wavenumber~\(-k\) in the sum (as is usual), then the complex sums for \(\cI^{IJ}_{1n}\) and~\(\cI^{JI}_{n1}\) are all real-valued, and since they are complex conjugates they must be equal, \(\cI^{IJ}_{1n}=\cI^{JI}_{n1}\)\,.
Hence, \cC~is symmetric and so this spectral coupling preserves self-adjoint symmetry.

\begin{figure}
\centering
\caption{\label{fig:homoExample1Uxt}%
example simulation in 1D space of heterogeneous diffusion~\eqref{eq:1Ddiff} on nine patches with spectral coupling~\eqref{eqs:1Dspec} and rather large size ratio \(r=0.3\) (for visibility).  
The listed diffusion coefficients~\(\kappa_{i+1/2}\) have period five, and each patch has \(n=5\) interior points as plotted.
The first five computed eigenvalues of the scheme's operator~\cL\ are expectedly close to~\(-k^2\), and then have a large spectral gap to the sub-patch modes.
}
\raisebox{-0.5\height}{\includegraphics{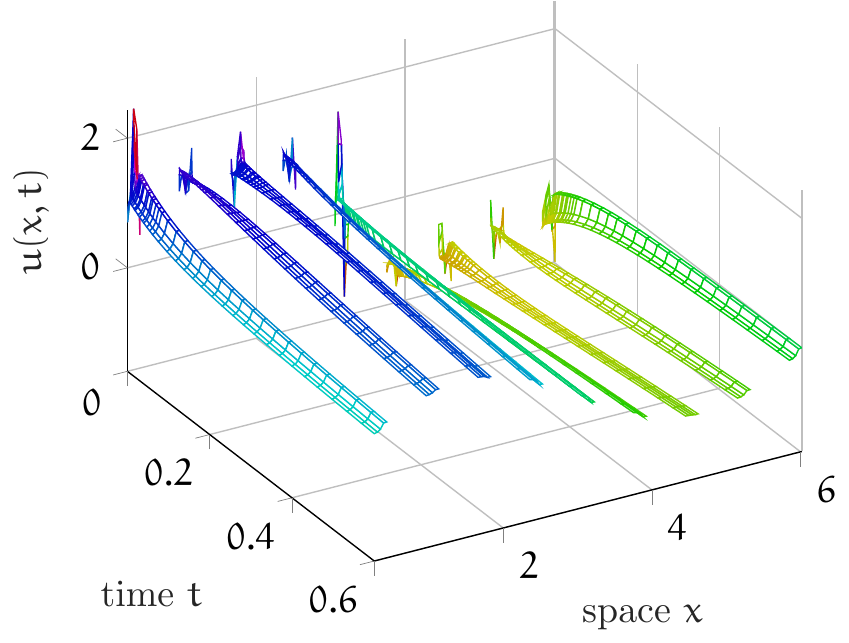}}
\(\begin{array}{r}
\kappa_{i+1/2}\\\hline
       3.965 \\
       2.531 \\
       0.838 \\
       0.331 \\
       7.275 \\
 \\\phantom{\vdots}
\end{array}\)
\(\begin{array}{r}
\text{eigenvalues}\\\hline
0\quad\\
      -0.9987 \\
      -3.9788 \\
      -8.8918 \\
      -15.654 \\
      -672.93 \\
      \vdots\quad
\end{array}\)
\end{figure}

\cref{fig:homoExample1Uxt} shows one example patch simulation in space-time of the heterogeneous diffusion~\eqref{eq:1Ddiff}.
In the initial condition at time \(t=0\) the ragged sub-patch structure is rapidly smoothed within each patch---the remaining sub-patch structure is due to the heterogeneous diffusion. 
Then the patches evolve over the shown macroscale time with the macroscale mode \(u\propto \sin x\) decaying slowest as appropriate.

\begin{remark}[Hilbert space generalisation]\label{remHilb}
Our discussion predominantly addresses the case where the field values~\(u_i\in\RR\).  
However, the arguments equally well apply to cases where the field values~\(u_i\) are in a Hilbert space, say denoted~\HH.
In that case the diffusivities~\(\kappa_{i+1/2}\) are to be interpreted as linear operators \(\kappa_{i+1/2}:\HH\to\HH\)\,, and the discourse appropriately rephrased---provided the operators are suitable.
Such a generalisation empowers much wider applicability of the results we establish, but for simplicity we mainly focus on the basic case of real~\(u_i\).
Nonetheless, \cref{sec:ens} introduces an ensemble of phase-shifted diffusions whose analysis requires the case of~\(u_i\) being in the Hilbert space of~\(\RR^p\), and similarly in \cref{sec:pd2D}.
\end{remark}

\subsection{Lagrangian spatial coupling of patches}
\label{sec:fwsc}

Spectral coupling constructs a single Fourier interpolant through all patches and then use this one interpolant to compute the edge values of each patch; thus the coupling matrix elements $\cC^{IJ}_{1n}$~and~$\cC^{IJ}_{n1}$ are nonzero for every patch $I$~and~$J$---it is a global coupling
An alternative local scheme is, for each patch, to construct an interpolant from neighbouring patches so that the inter-patch coupling occurs only over a finite local region of the spatial domain.
Such Lagrangian, or polynomial, coupling is the most common form to date \citep[e.g.,][]{Roberts13, Cao2015, Bunder2017}, but has always been constructed by interpolation of the fields (or field averages) from the centre of each patch---a scheme which does not generally maintain self-adjoint symmetry.
In contrast, our scheme to couple via next-to-edge values preserves self-adjoint symmetry.
Consequently, sensitive simulations such as the propagating waves of \cref{fig:homoWaveEdgyU2} preserve physically desirable properties, even with low-order coupling (here quadratic), and for propagation through this microscale heterogeneous medium.

\begin{figure}
\centering
\caption{\label{fig:homoWaveEdgyU2}patch simulation, with quadratic coupling, of `weakly damped' waves on a heterogeneous lattice illustrates the emergence of a coherent travelling macroscale wave from a noisy initial condition: in terms of the microscale difference operator~\(\delta_i\) (\cref{tblopids}), the system is \(\de t{u_i}=v_i\)\,, \(\de t{v_i}=\delta_i(\kappa_i\delta_i u_i)+0.02\,\delta_i^2v_i\) for log-normal~\(\kappa_i\).}
\includegraphics{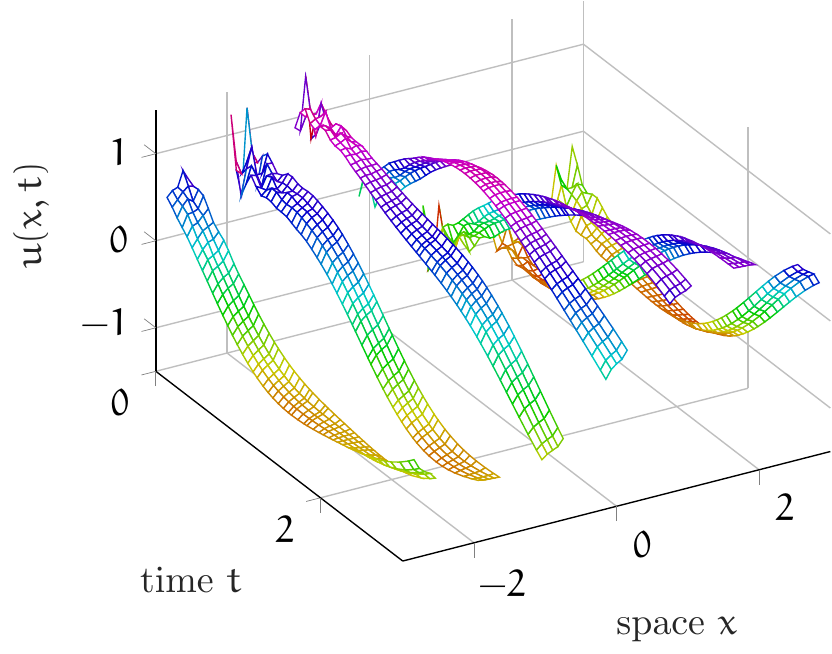}
\end{figure}

To develop the Lagrangian coupling, here we adapt the derivation of \cite{Bunder2017} to the  constraints identified by \cref{sec:sa1D} for self-adjoint coupling.
Define a macroscale step operator~$E$ which shifts a field by one macroscale step~$H$ to the right: $E u_i^I=u_i^{I+1}$.
Its inverse shifts by one macroscale step to the left: $E^{-1} u_i^I=u_i^{I-1}$. 
Define standard macroscale mean and difference operators, $\mu:=(E^{1/2}+E^{-1/2})/2$ and $\delta:=E^{1/2}-E^{-1/2}$, respectively.
These three operators are related by $\mu^2=1+\frac14\delta^2$ and $E^{\pm 1}=1\pm\mu\delta+\frac12 \delta^2$.
When operating on a field~$u_i^I$, the mean and difference operators involve patches to the left and right of patch~$I$; for example, $\mu\delta u^I_i=(u^{I+1}_i-u^{I-1}_i)/2$ and $\delta^2 u^I_i=u^{I+1}_i-2u_i^I+u^{I-1}_i$, and for every positive exponent~$K$  the operators~$\mu\delta^{2K-1}u^I_i$ and~$\delta^{2K}u^I_i$ involve patch values $u^I_i,u^{I\pm 1}_i,\ldots,u^{I\pm K}_i$.

To construct the Lagrangian self-adjoint coupling, we write the edge values in terms of known next-to-edge values for patches~$J$ near patch~$I$.
Define~$\Gamma$ as the number of coupled nearest neighbours to the left and right of each patch~$I$.
That is, patch~$I$ is coupled to the $2\Gamma+1$~closest patches, including itself, and wrapped periodically for the case of macroscale periodicity, thus forming an effective local coupling stencil of physical width~$(2\Gamma+1)H$ 
For patch~$I$, the right-edge~$u^I_{n+1}$ is a distance $h=nd=rH$ from the left-next-to-edge~$u^I_{1}$.
Consequently, in terms of the fractional macroscale shift~$E^r$, $u^I_{n+1}=E^{r}u^I_1$ determines the right-edge value by interpolating the left-next-to-edge values.
We expand the fractional shift, in powers of~$\delta$ up to order~$2\Gamma$, via $E^r=(1+\mu\delta+\tfrac12 \delta^2)^r$, where powers of~\(\mu\) are removed via the identity $\mu^2=1+\frac14\delta^2$.
Then the right-edge values are computed as
\begin{subequations}\label{eqs:ccboth}%
\begin{equation}
u^I_{n+1}=E^{r}u^I_1 \approx u^I_1+\sum_{k=1}^\Gamma\left(\prod_{\ell=0}^{k-1}(r^2-\ell^2) \right)\frac{+(2k/r)\mu\delta^{2k-1}+\delta^{2k}}{(2k)!}u^I_1\,.\label{eq:ccright}
\end{equation}
Similarly for the left-edge values, with the difference that the fractional shift is~\(E^{-r}\) from the right-next-to-edge, so that left-edge values are computed as 
\begin{equation}
u^I_{0}=E^{-r}u^I_n \approx u^I_n+\sum_{k=1}^\Gamma\left(\prod_{\ell=0}^{k-1}(r^2-\ell^2) \right)\frac{-(2k/r)\mu\delta^{2k-1}+\delta^{2k}}{(2k)!}u^I_n\,.\label{eq:ccleft}
\end{equation}
\end{subequations}
The highest-order operators are $\mu\delta^{2P-1}$ and~$\delta^{2P}$, and thus the above expressions involve patches~\(J\) where~$J$ is no more than $\Gamma$~from patch~$I$.
The coefficients of~$u^J_1,u^J_n$ in the interpolations~\eqref{eqs:ccboth} define the interpolation matrix elements~$\cI^{IJ}_{n1},\cI^{IJ}_{1n}$, respectively, as in~\eqref{eq:sacoup}, for patches $I$~and~$J$ no more than~$\Gamma$ apart. 
For greater distances between the two patches $\cI^{IJ}_{n1}=\cI^{IJ}_{1n}=0$\,.

\cref{sec:sa1D} shows that self-adjoint coupling requires symmetry $\cI^{IJ}_{n1}=\cI^{JI}_{1n}$ (all elements are real so there is no need to involve the complex conjugate).
The coupling conditions \eqref{eqs:ccboth} satisfy this symmetry constraint via the~$r$ dependence of the interpolation matrix elements: $\cI^{IJ}_{n1}(r)=\cI^{JI}_{n1}(-r)=\cI^{IJ}_{1n}(-r)=\cI^{JI}_{1n}(r)$\,.
For example, for only nearest neighbour coupling with $\Gamma=1$, the interpolations~\eqref{eqs:ccboth} give
\begin{align*}
u^I_{n+1}&=\tfrac{r}{2}(r-1)u^{I-1}_1+(1-r^2)u^I_1+\tfrac{r}{2}(r+1)u^{I+1}_1\,,\nonumber\\
u^I_{0}&=\tfrac{r}{2}(r+1)u^{I-1}_n+(1-r^2)u^I_n+\tfrac{r}{2}(r-1)u^{I+1}_n\,,
\end{align*}
so, for every patch~$I$, $\cI^{I,I-1}_{n1}=\frac{r}{2}(r-1)=\cI^{I-1,I}_{1n}$, $\cI^{I,I+1}_{n1}=\frac{r}{2}(r+1)=\cI^{I+1,I}_{1n}$, and $\cI^{II}_{n1}=1-r^2=\cI^{II}_{1n}$.
Thus the above derivation establishes the following lemma. 

\begin{lemma}\label{lemsafw}
For every order~\(\Gamma\), coupling patches with the Lagrangian interpolation~\eqref{eqs:ccboth} preserves the self-adjoint symmetry of heterogeneous systems in the form~\eqref{eq:1Ddiff}.
\end{lemma}

\subsubsection{Accuracy of Lagrangian patch coupling}

For Lagrangian coupling order $\Gamma\geq 4$ plots of simulations are visually similar to those for spectral coupling such as \cref{fig:homoExample1Uxt}.
To investigate the homogenisation accuracy of the patch scheme with Lagrangian coupling, we compute the eigenvalues of the microscale heterogeneous matrix operator~$\cL$ in~\eqref{eq:diffnew} and compare them to eigenvalues obtained from spectral coupling (\cref{sec:sasc}) with otherwise identical parameters.

The modes which dominate the emergent homogenised dynamics are those corresponding to the smallest magnitude eigenvalues.
The heterogeneous diffusion system~\eqref{eq:full1Ddiff} has nonpositive eigenvalues of approximately~$-k^2$ for wavenumbers $k=0,1,2,\ldots$ in our scenarios. 
For $N$~patches a patch dynamics simulation supports~$N$ `macroscale' eigen-modes, those with eigenvalues of small magnitude.
The remaining eigenvalues are of large magnitude, and represent rapidly decaying sub-patch modes that are of negligible interest.
Generally, there is a large gap between the microscale and macroscale eigenvalues, typically~\(\propto 1/r^2\).
In the scenarios reported here, typically the ratio between the large magnitude microscale eigenvalues and the small magnitude macroscale eigenvalues is of the order of~$100$.
So our focus here is assessing the accuracy of the first few smallest magnitude  eigenvalues compared to the homogenised dynamics.

\begin{figure}
\centering
\caption{\label{fig:homoEigenFixr1}%
Log-log plot of relative errors of the first five unique nonzero macroscale eigenvalues of~$\cL$ in~\eqref{eq:diffnew}  with coupling order $\Gamma=5$, domain size~$2\pi$, and the period five diffusion coefficients of \cref{fig:homoExample1Uxt}, but different numbers of patches~$N$ and patch sizes~$n$ while keeping the same microscale spacing $d=rH/n$ for all data.
The reference eigenvalues to the right are from spectral coupling with $r=0.1$\,.
The markers $\circ,\times,{+}$ correspond to patch size ratios $r=0.1,0.4,0.8$\,, respectively, and the solid lines join \(r=0.1\) data.
Dashed grey lines display the power law~$N^{-10}$, that is,~\(H^{10}\).
}
\includegraphics{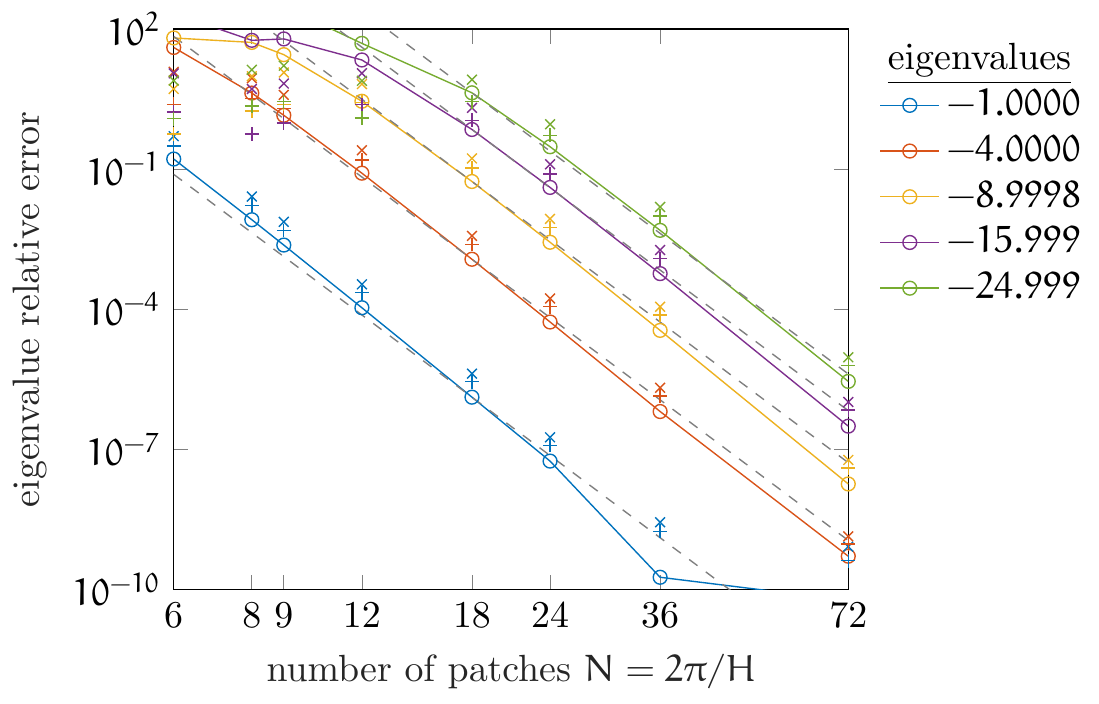}
\end{figure}

For a given order~\(\Gamma\) of interpolation, the main factor determining the accuracy of the computational homogenised simulation is the macroscale~$H$ which defines the inter-patch spacing.
As in classic numerical considerations, the reason is simply that closer spacing better interpolates the \emph{macroscale} structures.
We explore a domain of fixed length~\(2\pi\) so that the inter-patch spacing~\(H\) decreases as the number of patches~$N$ increases.
\cref{fig:homoEigenFixr1} illustrates that as the spacing decreases the patch scheme more accurately determines the lowest magnitude nonzero macroscale eigenvalues.
We do not show errors for the zero eigenvalue because this eigenvalue is always calculated to have magnitude no more than~$10^{-9}$, which is essentially zero (to round-off error). 
For the non-zero eigenvalues, \cref{fig:homoEigenFixr1} shows that their error decreases as the power-law~$\sim H^{10}\propto N^{-10}$ as expected for the eleven patch stencil width of the coupling order \(\Gamma=5\).

The solid lines of \cref{fig:homoEigenFixr1} highlight the data at fixed patch size ratio \(r=0.1\)\,.
That is, as the inter-patch spacing is decreased, the patch size is proportionally decreased.
The microscale lattice spacing~\(d\) is fixed for all of \cref{fig:homoEigenFixr1}, so the underlying heterogeneous microscale system is the same for all data in the figure. 
The figure also plots errors for other patch size ratios \(r=0.4,0.8\), and this data verifies that the errors in the patch scheme have only a weak dependence upon the patch size. 
That is, for computational efficiency, choose as small a patch as necessary to resolve \text{the microscale dynamics.}

\begin{figure}
\centering
\caption{\label{fig:homoEigen1}%
Relative errors of the first five nonzero macroscale eigenvalues obtained from different widths of inter-patch coupling.
Here the errors are for $N=20$ patches on a domain of~$2\pi$, with patch size \(n=5\),  and patch size ratio \(r=0.1\)\,.
The dashed grey lines approximate the data by error${}\approx 20 k^{\alpha\Gamma} e^{-2\alpha\Gamma}$ for constant $\alpha=1.9$ and wavenumber $k=1:5$\,.
Spectral coupling gives the reference eigenvalues list in the right-hand column.
The diffusion coefficients~\(\kappa_{i+1/2}\), with period five, are the same as those of \cref{fig:homoExample1Uxt}. 
}
\tikzsetnextfilename{Figs/eigenerrord5-20-10}
\includegraphics{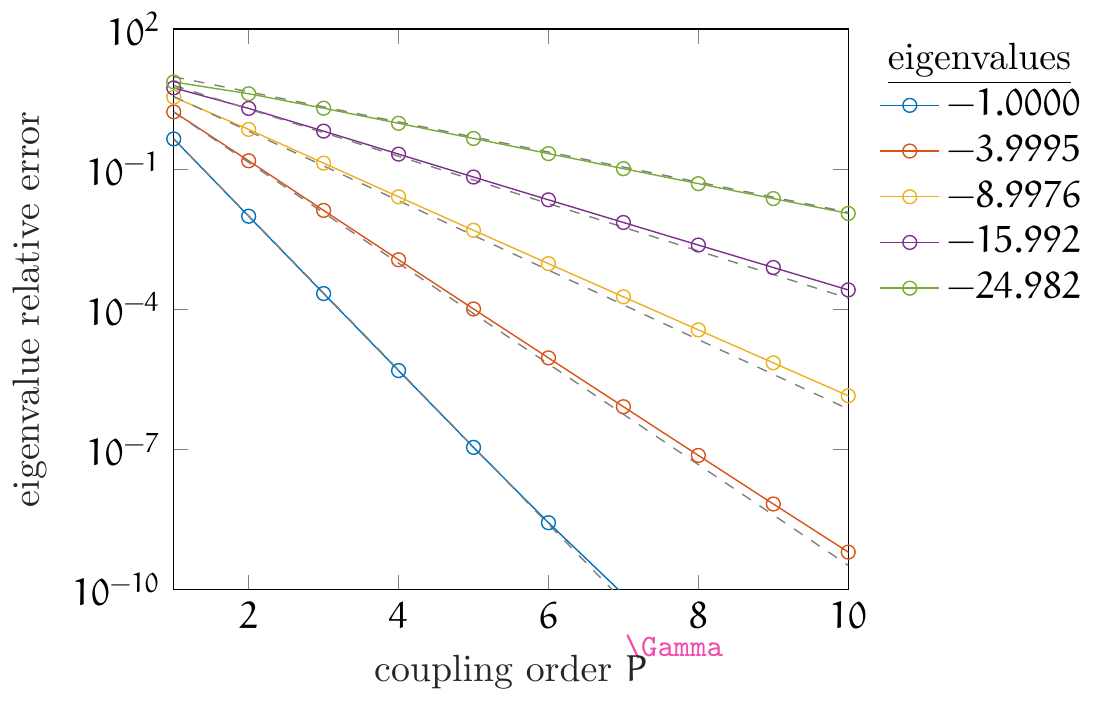}
\end{figure}

In application to some large-scale physical scenario we would require the patch scheme to resolve spatial structures on some macroscale, and so you choose the inter-patch spacing~\(H\) accordingly.
For a given domain this determines the number of patches.
Then choose an inter-patch coupling order~$\Gamma$, for the Lagrangian spatial coupling~\eqref{eqs:ccboth}, to suit your desired error in predictions by the patch scheme. 
\cref{fig:homoEigen1} demonstrates that the errors for the macroscale modes of the patch scheme decrease exponentially quickly in the order~$\Gamma$ of inter-patch coupling.
The lessened rate of the exponential decrease for higher wavenumbers, here \(\text{error}\propto\exp\big[-\alpha(2-\log k)\Gamma\big]\), is due to the smaller scale macroscale modes having fewer patches to resolve their structure.
\cref{fig:homoEigen1} is for $N=20$ patches with size ratio $r=0.1$\,, but other parameter choices produce much the same plot:  in the computational homogenisation of the patch scheme, relative errors generally decrease exponentially with the order~$\Gamma$ of Lagrangian inter-patch coupling.

\subsection{An ensemble removes periodicity limitation}
\label{sec:ens}

\cref{sec:sa1D} deduced that this new patch scheme preserves a self-adjoint heterogeneous system when the size~\(n\) of the patch is an integral multiple of the diffusivity period~\(p\).
This section proves that by using an ensemble of phase-shifts of the diffusivities, the patch scheme can still preserve self-adjoint symmetry \emph{without} requiring the integral multiple limitation.
As long recognised in Statistical Mechanics, a rigorous route to modelling is by considering an ensemble \cite[e.g.][]{vanKampen92, Sethna2010}.

We consider an ensemble with \(p\)-members of the heterogeneous system~\cref{eq:full1Ddiff} with each member of the ensemble distinguished by a different phase shift of the \(p\)-periodic diffusivities~\(\kappa_{i+1/2}\).
Importantly, do not think of this as an ensemble of a patch scheme for~\cref{eq:full1Ddiff}, but instead think of it as a patch scheme applied to an ensemble of~\cref{eq:full1Ddiff}.
In multiscale modelling ensembles constructed from microscale phase shifts are useful in many different contexts; for example, \citet{Runborg02} applied projective integration in a coarse bifurcation analysis of an evolution equation with spatially varying coefficients, with an ensemble constructed from phase shifts in \emph{time}, in contrast to the spatial phase shifts considered here.

To form the ensemble let \(u_{i,\ell}(t)\) be the field value at location~\(x_i\) in the \(\ell\)th~member of the ensemble, for \(\ell=0:p-1\).
These satisfy all phase-shifts of the diffusivities in the \ode{}s~\eqref{eq:full1Ddiff}, namely
\begin{equation}
d^2\partial_t u_{i,\ell}=\kappa_{i+\ell+\frac12}(u_{i+1,\ell}-u_{i,\ell})+\kappa_{i+\ell-\frac12}(u_{i-1,\ell}-u_{i,\ell}).
\label{eq:full1DdiffE}
\end{equation}
Throughout, the diffusivity subscripts are to be interpreted modulo their periodicity~\(p\). 
Form patches of this ensemble-system as in \cref{fig:1Dpatch} with \(u^I_{i,\ell}(t)\) denoting the evolving field values in the \(I\)th~patch at spatial location~\(x_i^I\).
\cref{fig:1Dcoupling43} illustrates five patches in the case of an ensemble of three members for patches of size \(n=4\) of a system with diffusivity period \(p=3\)\,.

\begin{figure}
\centering
\caption{\label{fig:1Dcoupling43}%
Five patches of the three member ensemble of phase-shifts in the case of diffusivity period $p=3$ and $n=4$ patch interior points:  the diffusivities are $\kappa_{1/2},\kappa_{3/2},\kappa_{5/2}$.
Lattice points and patches in the same column are at the same physical location. 
For the middle patch in each member, \(u\)-values at the edges  $i={\color{blue} 0},{\color{red} 5}$ (unfilled {\color{blue} triangles}\slash {\color{red} circles}) are interpolated from  \(u\)-values of next-to-edge points $i={\color{blue} 4},{\color{red} 1}$ (filled {\color{blue} triangles}\slash {\color{red} circles}) in neighbouring patches, but from a different member in order to preserve self-adjoint symmetry.}
\includegraphics{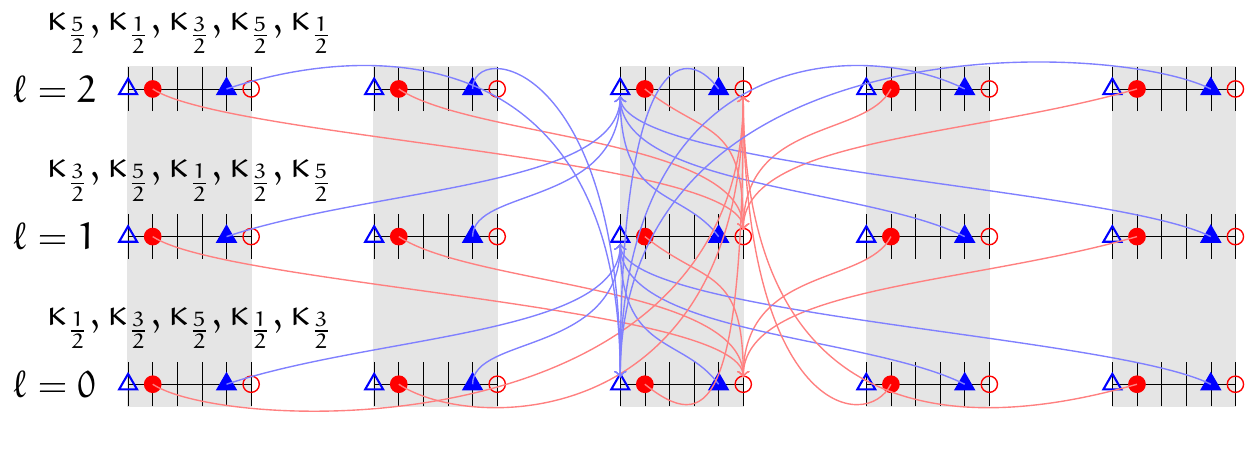}
\end{figure}

\cref{fig:1Dcoupling43} illustrates the proposed interpolation of edge values from next-to-edge values in the ensemble, and shows a tangle of dependencies.  
This inter-patch communication arises in the following way based upon the analysis and notation of the previous \cref{sec:sa1D,sec:sasc,sec:fwsc}.
The previous symbol~\(u^I_i\) here denotes the ensemble vector \((u^I_{i,0},\ldots,u^I_{i,p-1})\in\RR^P\).
The previous diffusivity~\(\kappa^I_{i\pm1/2}\) is here to denote the diffusivity matrix \(\diag(\kappa_{i\pm1/2},\ldots,\kappa_{i+p-1\pm1/2})\in\RR^{p\times p}\) (subscripts modulo~\(p\) as always).
Then the patch scheme applied to the heterogeneous ensemble~\eqref{eq:full1DdiffE} is symbolically the \ode{}s~\eqref{eq:1Ddiff} but here interpreted as matrix-vector \ode{}s instead \text{of scalar \ode{}s.}

Consequently, all the arguments and results of the previous \cref{sec:sa1D,sec:sasc,sec:fwsc} apply here also.
Except that we now have extra freedom in the patch-coupling interpolation matrix.
Previously the symbol~\(\cC^{IJ}_{ij}\) was a scalar, but here represents a \(p\times p\) block in the ensemble system's \(nNp\times nNp\) matrix. 
Thus we have the freedom to choose the crucial interpolation blocks \(\cC^{IJ}_{n1},\cC^{IJ}_{1n}\) in non-diagonal form.
The tangle of inter-patch communication in \cref{fig:1Dcoupling43} represents a non-diagonal~\(\cC^{IJ}_{n1},\cC^{IJ}_{1n}\).

We choose~\(\cC^{IJ}_{n1},\cC^{IJ}_{1n}\) to preserve self-adjoint symmetry, and \cref{A1,A2}.  
In the \(\ell\)th~member of the ensemble, the \(\ell\)th~row of \cref{fig:1Dcoupling43}, has diffusivity~\(\kappa_{n+\ell+1/2}\) (subscript~\(\mod p\)) at the right-edge of patches.
So the interpolation from its right-next-to-edges is chosen to determine the left-edge values of member~\((n+\ell \mod p)\) because it has the same diffusivity~\(\kappa_{n+\ell+1/2}\) at its left-edge.
Correspondingly in reverse for interpolation from left-next-to-edges to right-edge values.
Hence, setting the \(p\times p\) matrix~\(K\) to zero except for \(K_{\ell+1,(\ell+n\mod p)+1}:=\kappa_{\ell+1/2}\)\,, $\ell=0:p-1$\,, let's choose 
\begin{equation}
\cC^{IJ}_{1n}:=K\cI^{IJ}_{1n}
\quad\text{and}\quad 
\cC^{IJ}_{n1}:=K^\dag\cI^{IJ}_{n1}
\label{eq:enscoup}
\end{equation} 
in terms of the scalar interpolation coefficients~\(\cI^{IJ}_{ij}\) of the  interpolation schemes of \cref{sec:sasc,sec:fwsc}.
Because of this choice of~\(K\), the above argument establishes the following lemma.
\begin{lemma}
The inter-patch coupling with~\eqref{eq:enscoup} preserves self-adjoint symmetry in the ensemble of heterogeneous diffusion systems~\eqref{eq:full1DdiffE}.
\end{lemma}

\begin{figure}
\centering
\caption{\label{fig:homoExampleEns}
the ensemble-mean field~\(\bar u(x,t)\) in a simulation of an ensemble of heterogeneous diffusion~\eqref{eq:full1DdiffE} on nine patches with spectral coupling~\eqref{eqs:1Dspec} and patch-ratio \(r=0.3\), for comparison with \cref{fig:homoExample1Uxt}.  
The diffusion coefficients~\(\kappa_{i+1/2}\) with period \(p=5\) are as in \cref{fig:homoExample1Uxt}, but here there are just \(n=4\) patch interior points, as plotted.
}
    \includegraphics{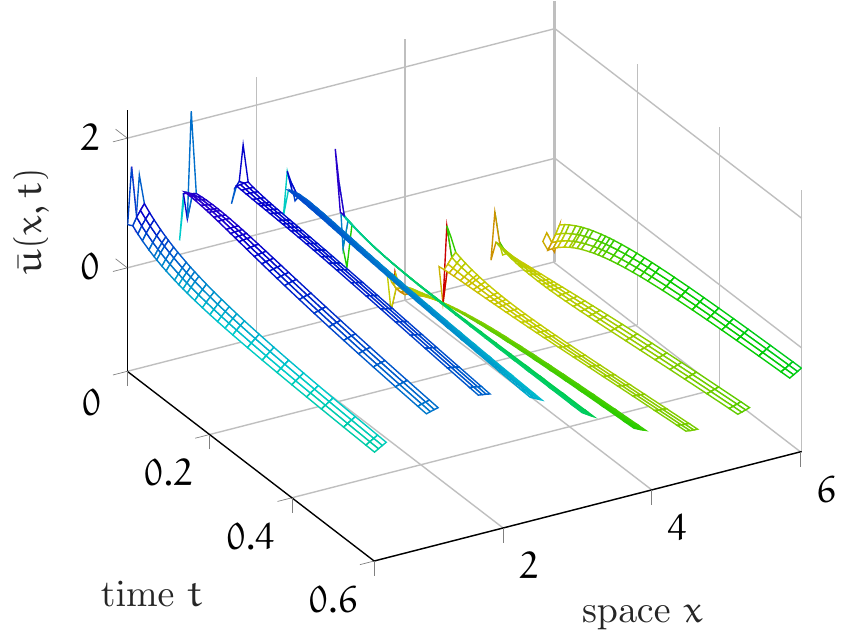}
\end{figure}

As an example, \cref{fig:homoExampleEns} plots a simulation of our patch scheme applied to heterogeneous diffusion with diffusivity period $p=5$ and the same diffusion coefficients as listed in \cref{fig:homoExample1Uxt}. 
With the exception of the number of patch interior points, which here is \(n=4\), the parameters of \cref{fig:homoExampleEns} are the same as in \cref{fig:homoExample1Uxt}. 
At each spatial point~\(x_i^I\) there are five ensemble values~\(u^I_{i,0},\ldots,u^I_{i,4}\) so \cref{fig:homoExampleEns} plots the ensemble-mean \(\bar u(x^I_i,t):=\frac1p\sum_{\ell=0}^{p-1}u^I_{i,\ell}(t)\).
In the simulation of \cref{fig:homoExampleEns}, all five members of the ensemble were given the same initial condition, and this initial condition is the same as that for \cref{fig:homoExample1Uxt} but with a different random component. 
Because of the average over the ensemble, \cref{fig:homoExampleEns} does not exhibit the rough microscale structure shown in \cref{fig:homoExample1Uxt} that arises from just one phase of the heterogeneous diffusivity.

To verify the accuracy of the patch scheme applied to the ensemble of phase-shifts, we investigated the accuracy of the small-magnitude eigenvalues that correspond to the macroscale modes in the computational scheme.
Over a range of coupling orders~$\Gamma$, we plotted the relative errors of nonzero macroscale eigenvalues of ensemble matrix~$\cL$ for Lagrangian spatial coupling~\eqref{eqs:ccboth} (relative to the spectral eigenvalues~\eqref{eqs:1Dspec}).
The plots were graphically indistinguishable from that of \cref{fig:homoEigen1}---so are not reproduced here.
The crucial difference is that here we used an ensemble of patches of size $n=4$ points---a size which is not an integral multiple of the diffusivity period \(p=5\). 
Evidently the patch scheme applied to an ensemble for general patch size of \(n\)~points appears to be just as accurate as for the well-known special case of \(n\) being an integral multiple of the period~\(p\).
\cref{sec:theory} establishes this \text{accuracy in general.}

Invoking this ensemble of all phase-shifts of the diffusivities allows any size patch, any~$n$, while appearing to maintain the full accuracy of the computational homogenisation supplied by the patch scheme, albeit with an increase in the computation.
We conjecture that if the heterogeneous diffusivities is random, with no period, then an ensemble of realisations will provide more accurate predictions than a single realisation.

\section{The patch scheme is consistent to high-order}
\label{sec:theory}

This section develops theoretical support for the accuracy of the patch scheme with self-adjoint coupling. 
But before we discuss patches, we reconsider the heterogeneous diffusion~\eqref{eq:full1Ddiff} and its ensemble of all phase shifts~\eqref{eq:full1DdiffE}.
\cref{sssech} establishes that the ensemble system describes the correct macroscale homogenised behaviour.
\cref{sec:con} then proves that the patch scheme with self-adjoint coupling of the ensemble is consistent with the ensemble system, and hence with the heterogeneous diffusion~\eqref{eq:full1Ddiff}.

In this and the next section, the term ``consistent'' means ``consistent with arbitrarily high-order in~\(H\)'', unless otherwise specified.

We use consistency to assess accuracy of the patch scheme because the usual numerical analysis, of the finite element and finite volume methods, involves integrals over space which in a patch scheme would be integrals over much uncomputed space and so appear to be inappropriate for a patch scheme. 

\cref{sec:ens} creates the ensemble of phase shifts by directly `stacking vertically' each member in the ensemble, as illustrated by \cref{fig:1Dcoupling43}.
This appears best for computation.
Best because a user codes the lattice dynamics in the `horizontal' direction as usual, and then the tangle of inter-patch communication (\cref{fig:1Dcoupling43}) may be managed by a generic patch function as in our toolbox \cite[]{Roberts2019b}.

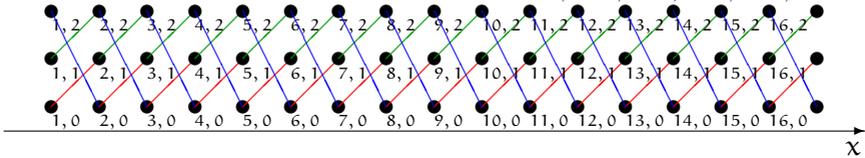
\begin{figure}
\centering
\caption{\label{figembedshifts}embed the heterogeneous diffusion~\eqref{eq:full1Ddiff} into \(p\)~replications of itself, here for the case \(p=3\).
The circles denote values \(v_{i,\ell}(t)\) for the example labelled values~\(i,\ell\), and the coloured lines denote the heterogeneous diffusion between~\(v_{i,\ell}\) that both replicates~\eqref{eq:full1Ddiff}, but is now spatially homogeneous, in~\(x\), on the lattice.
Over~$x$, the fields~\(v_{i,\ell}(t)\) of any one of the $p=3$ phases follow the coloured diagonals representing the heterogeneous diffusion for that phase, for example, one phase has fields \(v_{1,1}, v_{2,2},v_{3,3},v_{4,1},v_{5,2},\ldots\)}
\setlength{\unitlength}{0.01\linewidth}
\def\dx{5}\def\p{3}\def\pm{2}
\begin{picture}(90,16)
\put(0,2.5){\vector(1,0){90}}
\put(88,0){\(x\)}
\setcounter{j}0
\multiput(5,5)(0,\dx){\p}{%
    \multiput(0,0)(\dx,0){17}{\circle*{1.5}}
    \setcounter{i}0
    \multiput(0,-2)(\dx,0){16}{\stepcounter{i}%
        \tiny$\arabic{i},\arabic{j}$}
    \multiput(0,0)(\dx,0){16}{\ifcase\value{j}%
        \color{red}\line(1,1){\dx}
        \or\color{green!60!black}\line(1,1){\dx}
        \or\color{blue}\line(1,-\pm){\dx}
        \fi}
    \stepcounter{j}
}
\end{picture}
\end{figure}

However, for theoretical analysis it appears best, and we take this route here, to form the ensemble so that each member of the ensemble is `wrapped diagonally', as illustrated by \cref{figembedshifts} for the case of periodicity \(p=3\)\,.
In terms of~\(u_{i,\ell}\) of \cref{sec:ens},
define \(v_{i,\ell}(t):=u_{i,(i-\ell \mod p)}\). 
That is, \(v_{i,\ell}\)~is a field value at position~\(x_i\), for \(\ell=0:p-1\)\,, and \(\phi=i-\ell\mod p\) identifies the physical phase-shift for the member of the ensemble.
According to~\eqref{eq:full1DdiffE} these evolve in time according to the \ode{}s
\begin{equation}
d^2\partial_t v_{i,\ell}=\kappa_{\ell+1/2}(v_{i+1,\ell+1}-v_{i,\ell})+\kappa_{\ell-1/2}(v_{i-1,\ell-1}-v_{i,\ell})\,,
\label{eq:mbd1Ddiff}
\end{equation}
where we adopt the convention that both the subscript of~\(\kappa\) and the second subscript of~\(v\) is always interpreted modulo the microscale periodicity~\(p\).
Consequently, \(u_i(t):=v_{i,i+\phi}\) then satisfies the original diffusion~\eqref{eq:full1Ddiff}, but with the diffusivities phase shifted by~\(\phi\).
That is, the system~\eqref{eq:mbd1Ddiff} for~\(v_{i,\ell}\) captures the \(p\)~microscale-phase shifted versions~\eqref{eq:full1DdiffE} of the original heterogeneous diffusion~\eqref{eq:full1Ddiff}.
As in macroscale modelling \cite[\S2.5]{Roberts2013a}, the big advantage of~\eqref{eq:mbd1Ddiff} is that the heterogeneity only occurs in~\(\ell\) and is independent of, \emph{homogeneous} in, the  spatial index~$i$. 
This homogeneity is unlike systems \eqref{eq:full1Ddiff,eq:full1DdiffE} where the heterogeneity explicitly varies with the spatial index.
The homogeneity of~\eqref{eq:mbd1Ddiff} in spatial index~\(i\) is crucial in \text{developing theory.}

\subsection{An ensemble has the correct homogenisation}
\label{sssech}
The aim of this subsection is to establish that the solutions of the ensemble system~\eqref{eq:mbd1Ddiff} track solutions of an effective `homogenised' \pde\ \(V_t=\cK_2 V_{xx}\) for a field~\(V(x,t)\) and an effective diffusivity~\(\cK_2\).

We analyse the long-time dynamics via the Fourier transform in space \cite[\S7.2 and Exercise~7.5]{Roberts2014a}.
For grid-scaled spatial Fourier wavenumber~\(k\) we seek solutions \(v_{i,\ell}(t)=\int_{-\pi}^\pi e^{\i ki}\tv_\ell(k,t)\,dk\) (for \(\i:=\sqrt{-1}\), distinct from lattice index~\(i\)).
Because of the linear independence of~\(e^{\i ki}\), \eqref{eq:mbd1Ddiff}~becomes, for every wavenumber~\(k\),
\begin{equation}
d^2\partial_t \tv_{\ell}=\kappa_{\ell+\frac12}(e^{\i k}\tv_{\ell+1}-\tv_{\ell})+\kappa_{\ell-\frac12}(e^{-\i k}\tv_{\ell-1}-\tv_{\ell}),
\quad \ell=0:p-1\,.
\label{eq:mbd1DdiffF}
\end{equation}
This system has a subspace of equilibria for wavenumber \(k=0\) and \(\tv_\ell={}\)constant.
Since the problem is linear in~\(\tv\), without loss of generality we analyse the case of equilibrium \(\tv_\ell=0\) \cite[\S7.2]{Roberts2014a}.
When wavenumber \(k=0\) the system~\cref{eq:mbd1DdiffF} has one eigenvalue of zero, and \((p-1)\)~negative eigenvalues \(\lambda\leq-\beta\) for bound \(\beta=2\pi^2\min_\ell\kappa_{\ell+1/2}/(p^2d^2)\).
Notionally adjoining \(dk/dt=0\), we deduce there exists an emergent slow manifold of~\(\cref{eq:mbd1DdiffF}\) for a range of small~\(k\), and globally in~\(\tv\) \cite[]{Carr81}.

Straightforward construction \cite[][as in Exercise~7.5]{Roberts2014a} leads to the evolution on the slow manifold, in terms of some chosen parameter~\(\tV(k,t)\).
Here we describe the evolution on the slow manifold in terms of the mean Fourier component \(\tV(k,t):=\frac{1}{p}\sum_\ell\tv_\ell(t)\). We seek the evolution of~\(\tV\)\ in terms of a power series in small wavenumber~\(k\), up to some specified order.
For given diffusivities and periodicity~$p$, the computer algebra in \cref{appmh1dd} efficiently derives the evolution via an iteration, based upon the residual of~\eqref{eq:mbd1DdiffF}, from the initial approximation that $\partial_t\tV\approx0$ and $\tv_\ell\approx\tV$ for every $\ell=0:p-1$\,.
Subsequent iterations provide corrections in terms of $\tV$~and~powers of~$k$.
In general, the dynamical evolution on the slow manifold is then of the form 
\begin{equation}
d^2\partial_t \tV=-k^2\cK_2\tV+k^4\cK_4\tV+\cdots\,,
\label{eq:echft}
\end{equation}
for \(\cK_2:=n/\big(\sum_\ell\kappa_{\ell+1/2}^{-1}\big)\), and some complicated~\(\cK_4\).
As an example, for periodicity \(p=3\) and \(\kappa_{\ell+1/2}=\ell+1\)\,; \cref{appmh1dd} computes  \(\cK_2=\frac{18}{11}\) and \(\cK_4=\frac{675}{2662}\)\,.

We obtain a physical space \pde\ for the macroscale by integrating over small wavenumbers \cite[][Exercise~7.5]{Roberts2014a}.
Let \(k_c\)~denote a cut-off wavenumber, suitable to capture the macroscales of interest, and define \(V(x,t):=\int_{-k_c}^{k_c}e^{\i kx/d}\tV(k,t)\,dk\)\,.
Upon correspondingly integrating~\eqref{eq:echft}, we deduce the emergent slow manifold dynamics is equivalently governed by the `homogenised' \pde
\begin{equation}
\D tV=\cK_2\DD xV+\cK_4d^2\Dn x4V+\cdots\,.
\label{eq:slowman}
\end{equation}
The coefficient~\(d^2\cK_4\) indicates that our approach supports not only the classic diffusion homogenisation as the microscale \(d\to0\), but also establishes corrections at finite~\(d\).
Such higher-order corrections are needed in some homogenisation applications \cite[e.g.,][]{Cornaggia2020}.
We also contend that the technique of \cite{Roberts2013a} would extend to provide, at finite scale separation~\(d\), a rigorous error formula for any finite truncation of this asymptotic series, and would also do so for nonlinear systems.

This approach rigorously establishes that the classic homogenisation of the ensemble system~\cref{eq:mbd1Ddiff} is the leading approximation to the evolution of long wavelength structures on the lattice.
Moreover, the centre manifold emergence theorem \cite[e.g.,][\S4.3]{Roberts2014a} assures us that all solutions approach this homogenisation on a cross-period diffusion time\({}\propto1/\beta\).

\subsection{High-order consistency of the patch scheme}
\label{sec:con}

This section establishes that when the phase-shifted ensemble system~\cref{eq:mbd1Ddiff} (\cref{figembedshifts}) is restricted to patches, with self-adjoint coupling as in \cref{fig:embedn3}, the resulting patch system maintains consistency with~\cref{eq:mbd1Ddiff}.
The analysis shows that any errors arising in the scheme are due to the choice of coupling condition. 
The advantage of the ensemble~\eqref{eq:mbd1Ddiff} for the patch scheme is that the microscale system within patches is homogeneous in the \(x\)-index~\(i\); the heterogeneity in the original diffusion~\cref{eq:full1Ddiff} only appears here in the cross-section indexed by~$\ell$ (\cref{fig:embedn3}).

\begin{figure}
\setlength{\unitlength}{0.0097\linewidth}
\def\Hx{35} \def\dx{4.5}
\def\p{3}\def\pm{2}
\def\n{4}\def\np{5}\def\npp{6}
\def\hx{18}%
\centering
\caption{\label{fig:embedn3} example of patch scheme (\cref{fig:1Dpatch}) with \(n=\n\) points in each patch, for an ensemble with heterogeneity of period \(p=\p\) (\cref{figembedshifts}).  
These homogeneous patches are coupled, as in \cref{fig:1Dcoupling}, by interpolation to the {left}\slash{right} edge of a patch from the next-to-edge points on the {right}\slash{left} of nearby patches and always to points with the same~$\ell$. 
For each $\ell=0:p-1$\,, the value~$v^I_{\np,\ell}$ is interpolated from values~$v^J_{1,\ell}$ from nearby patches~\(J\), and $v^I_{0,\ell}$ is interpolated from~$v^J_{\n,\ell}$.
}
\begin{picture}(102,20)
\put(2,0){%
\put(0,2){\vector(0,1){17}}\put(-2,17){$\ell$}
\setcounter{j}0
\multiput(-2,4.5)(0,\dx)\p{\small$\arabic{j}$\stepcounter{j}}
\put(0,2.5){\vector(1,0){100}}
\put(98,0){\(x\)}
\setcounter{I}0
\multiput(2,5)(\Hx,0)3{\stepcounter{I}%
  \put(8,-5){$\ifcase\value{I}\or I-1\or I\or I+1\fi$}
  \setcounter{j}0
  \multiput(0,0)(0,\dx){\p}{%
    \setcounter{i}0
    \multiput(0,0)(\dx,0){\npp}{%
        \ifnum\value{i}=0\circle{1.5}
        \else\ifnum\value{i}=\np\circle{1.5}
        \else\circle*{1.5}
        \fi\fi%
        \stepcounter{i}}
    \setcounter{i}0
    \multiput(0,-2)(\dx,0){\npp}{%
        \tiny$\arabic{i},\arabic{j}$%
        \stepcounter{i}}
    \multiput(0,0)(\dx,0){\np}{\ifcase\value{j}%
        \color{red}\line(1,1){\dx}
        \or\color{green!50!black}\line(1,1){\dx}
        \or\color{blue}\line(1,-\pm){\dx}
        \fi}
  \stepcounter{j}
  }
}
\put(2,5){\multiput(\Hx,0)(\Hx,0)2{\color{gray}
    \multiput(-8,-1)(0,\dx)\p{$\longleftrightarrow$}
}}
\put(4.25,16){\vector(1,0){\hx}\vector(-1,0){\hx}}
\put(8,17){$h=nd$}
\put(2,16){\put(\Hx,0){\vector(1,0){\Hx}\vector(-1,0){\Hx}}
    \put(\Hx,0){\put(17,1){$H$}}
}}
\end{picture}
\end{figure}
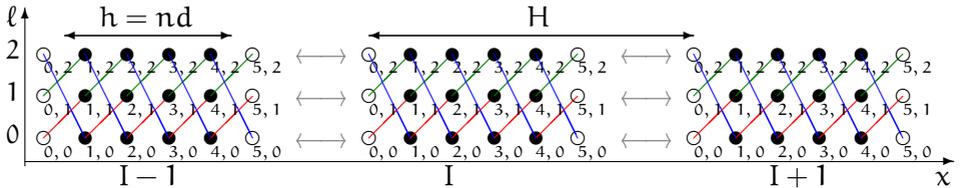

Corresponding to \cref{sec:pd1D}, let the macroscale spacing in~\(x\) of the patches be~\(H\)---much larger than the microscale lattice spacing~\(d\).
Let the patches have \(n\)~interior points so the patches are of width \(h:=nd\) in~\(x\), and let the ratio \(r:=h/H\).
Also let \(v^I_{i,\ell}(t)\)~to denote the field value at the \((i,\ell)\)th~point of the lattice within the \(I\)th~patch (\cref{fig:embedn3}).
For generality, let the vector~\(\vv(t)\in\RR^p\) denote the vector of \(p\)~values at each~\(x\), and then for patches the vector \(\vv^I_i:=(v^I_{i,0},\ldots,v^I_{i,p-1})\).
That is, here the vector~\vv\ corresponds to the scalar~\(u\) of \cref{sec:pd1D}.
We now use shift, difference, and mean operators as defined by \cref{tblopids} to express and analyse this patch scheme for the ensemble.

\begin{table}
\caption{\label{tblopids}Useful operator identities based upon the shift~\(E\) which shifts an operand to the right by some `distance'~\(\Delta\)  \protect\cite[p65, e.g.]{npl61}.
A subscript on any of the listed operators defines the distance~$\Delta$.
For example, \(E_i\)~increases the index~\(i\) by one (distance \(\Delta=d\)),  and \(E_x\) increases the position~\(x\) by~$H$  (distance \(\Delta=H\)).
}
\begin{equation*}
\begin{array}{rlrl}
\hline\vphantom{E^{1^1}}
&Eu(x):=u(x+\Delta)
&&E^{\pm1}=1\pm\mu\delta+\tfrac12\delta^2
\\&\delta:=E^{1/2}-E^{-1/2}
&&\mu:=\tfrac12(E^{1/2}+E^{-1/2})
\\&\delta=2\sinh(\Delta\partial_x/2)
&&\mu=\cosh(\Delta\partial_x/2)
\\&\Delta\partial_x=2\asinh(\delta/2)
&&\mu^2=1+\tfrac14\delta^2
\\[0.5ex]\hline
\end{array}
\end{equation*}
\end{table}

\begin{lemma}\label{lemlatvec}
The ensemble system~\cref{eq:mbd1Ddiff} is of the vector form
\begin{equation}
\partial_t\vv_i=\delta_i\big(\Kp\ E_i^{1/2}-\Km E_i^{-1/2}\big)\vv_i+D\vv_i\,,\quad  i=1:n\,,
\label{eqlatvec}
\end{equation}
for the three \(p\times p\) matrices~\(D\), \(\Kp\) and~\(\Km\). 
\end{lemma}

\begin{proof} 
Let's rewrite the ensemble lattice \ode{}s~\cref{eq:mbd1Ddiff} for patch interior indices $i=1:n$\,:
\begin{align*}
d^2\partial_tv_{i,\ell}
&=\kappa_{\ell+1/2}(v_{i+1,\ell+1}-v_{i,\ell+1})
+\kappa_{\ell-1/2}(v_{i-1,\ell-1}-v_{i,\ell-1})
\\&\quad{}
+\kappa_{\ell+1/2}(v_{i,\ell+1}-v_{i,\ell})
+\kappa_{\ell-1/2}(v_{i,\ell-1}-v_{i,\ell})
\\&=\delta_i\big(\kappa_{\ell+1/2}v_{i+1/2,\ell+1}
-\kappa_{\ell-1/2}v_{i-1/2,\ell-1}\big)
+\delta_\ell(\kappa_\ell\delta_\ell v_{i,\ell})
\\&=\delta_i\big(\kappa_{\ell+1/2}E_i^{1/2}v_{i,\ell+1}
-\kappa_{\ell-1/2}E_i^{-1/2}v_{i,\ell-1}\big)
+\delta_\ell(\kappa_\ell\delta_\ell v_{i,\ell})
\,,
\end{align*}
on applying operator definitions given in \cref{tblopids}.
Define the cell-diffusion \(p\times p\) matrix~\(D\) to encode the operator~\(\delta_\ell(\kappa_\ell\delta_\ell\cdot)/d^2\), and the shifted-diffusivity \(p\times p\) matrix~\(K\) to be zero except \(K_{i,i+1}:=\kappa_{i+1/2}/d^2\) for $i=1:p-1$ and \(K_{p,1}:=\kappa_{1/2}/d^2\).
Then the above \ode{}s become the vector system~\eqref{eqlatvec}
with matrices \(\Kp:=K\) and \(\Km:=K^T\).
\end{proof}

Since \cref{lemlatvec} establishes that~\cref{eq:mbd1Ddiff} and~\eqref{eqlatvec} are the same,
we propose a patch scheme for the heterogeneous diffusion system~\cref{eq:mbd1Ddiff} of the form 
\begin{equation}
\text{solve \eqref{eqlatvec} for \(\vi\) instead of \(\vv_i\), with}
\quad
\vi[n+1]=E_x^r\vi[1]\,,\quad 
\vi[0]=E_x^{-r}\vi[n]\,,
\label{eqlatc}
\end{equation}
as the inter-patch coupling via the edge-values. 
In the inter-patch coupling of~\eqref{eqlatc}, since the shift~$E_x$ describes a macroscale shift of~$H$, the shift~$E^r_x$ describes a macroscale fractional shift of $rH=nd$\,; that is,  $E^r_x$~describes a shift over the width of a patch.
So $E^r_x\vi[1]$ shifts each element from the left-next-to-edge point to the corresponding element on the right edge~$\vi[n]$, and similarly, $E_x^{-r}\vi[n]$ shifts from the right-next-to-edge value to the left edge~$\vi[0]$.
The following \cref{thm:psc} justifies the patch scheme~\eqref{eqlatc}.

\begin{theorem} \label{thm:psc}
The macroscale of the patch scheme~\cref{eqlatc} 
is consistent with the microscale dynamics of~\cref{eqlatvec} over the entire spatial domain, and hence with the ensemble~\eqref{eq:mbd1Ddiff}.
\end{theorem}

This theorem may appear almost vacuous as the \ode{}s~\cref{eqlatvec} are common to both parts of the claim.
However, the distinction is that the ``patch scheme~\cref{eqlatc}''  has the \ode{}s~\cref{eqlatvec} holding only inside small, well-separated, patches of the spatial domain, whereas the ``dynamics of~\cref{eqlatvec}'' are to hold on a lattice over the entire spatial domain.
With very different domains, they are two very different dynamical systems, and so the following proof is deeper than may be first appear necessary.

{\def\proofname{Proof of \cref{thm:psc}}
\begin{proof}
Consider the vector system~\cref{eqlatvec} on patches coupled by the interpolation~\cref{eqlatc}.
Using \cref{tblopids}, the system
\begin{align*}&
\partial_t\vi=\delta_i\big(\Kp E^{1/2}_i-\Km E^{-1/2}_i\big)\vi+D\vi
\\&
\iff \partial_t\vi-D\vi
=2\sinh[d\partial_x/2]\big(\Kp E^{1/2}_i-\Km E^{-1/2}_i\big)\vi
\\&
\iff (\partial_t-D)\vi
=2\sinh\big[\tfrac1n\asinh(\bar\delta/2)\big]\big(\Kp E^{1/2}_i-\Km E^{-1/2}_i\big)\vi
\end{align*}
for centred difference \(\bar\delta:=E_i^{n/2}-E_i^{-n/2}=2\sinh(nd\partial_x/2)\), as the difference is over a distance \(\Delta=h=nd\)\,.
Invoking cognate steps to those of the proof by \cite{Roberts2011a}, let's invert the operator function \(f(\bar\delta):=2\sinh\big[\tfrac1n\asinh(\bar\delta/2)\big]\) and write the above dynamical equation as
\begin{align*}
&f^{-1}(\partial_t-D)\vi
\\&
=\bar\delta\big(\Kp E^{1/2}_i-\Km E^{-1/2}_i\big)\vi
\\&
=E_i^{n/2}\big(\Kp E^{1/2}_i-\Km E^{-1/2}_i\big)\vi
-E_i^{-n/2}\big(\Kp E^{1/2}_i-\Km E^{-1/2}_i\big)\vi
\\&
=\Kp\vi[i+n/2+1/2]-\Km\vi[i+n/2-1/2]
-\Kp\vi[i-n/2+1/2]+\Km\vi[i-n/2-1/2]\,.
\end{align*}
Now evaluate this equation at the patch mid-point, \(i=n/2+1/2\) (a virtual mid-point when size~\(n\) is even), and here define the macroscale \(\Vv^I(t):=\vi[n/2+1/2]\in\RR^p\).
Hence
\begin{align*}
&f^{-1}(\partial_t-D)\Vv^I
\\&
=\Kp\vi[n+1]-\Km\vi[n]
-\Kp\vi[1]+\Km\vi[0]
\\&
=\Kp E_x^r\vi[1]-\Km\vi[n]-\Kp\vi[1]+\Km E_x^{-r}\vi[n]
\quad(\text{by edge interpolation})
\\&=\Kp(E_x^r-1)\vi[1]-\Km(1-E_x^{-r})\vi[n]
\\&=(E_x^{r/2}-E_x^{-r/2})(\Kp E_x^{r/2}\vi[1]-\Km E_x^{-r/2}\vi[n])
\quad(\text{by commutativity}).
\end{align*}
The above right-hand side is on the macroscale because it involves the macroscale inter-patch interpolation via the operator~\(E_x\). 
To compare with the original system, we transform this macroscale identity back to its equivalent on the microscale lattice.
To do so, notionally evaluate over the microscale the smooth macroscale interpolation underlying the operators~\(E_x^{\pm r/2}\) so that for the interpolated field we have \(E_x^{\pm r/2}=E^{\pm n/2}_i\).
Hence, for the smooth macroscale field from the patch scheme, the above identity becomes
\begin{align*}&
f^{-1}(\partial_t-D)\Vv^I
\nonumber\\&
=(E_i^{n/2}-E_i^{-n/2})(\Kp E_i^{n/2}\vi[1]-\Km E_i^{-n/2}\vi[n])
\nonumber\\&
=\bar\delta(\Kp\vi[n/2+1]-\Km\vi[n/2])
=\bar\delta(\Kp E_i^{1/2}-\Km E_i^{-1/2})\vi[n/2+1/2]
\nonumber\\&
=\bar\delta(\Kp E_i^{1/2}-\Km E_i^{-1/2})\Vv^I\,.
\end{align*}
Now revert the function~\(f^{-1}\), recalling \(f(\bar\delta)=\delta_i\)\,, to deduce that
\begin{align}&
(\partial_t-D)\Vv^I=f(\bar\delta)\big[(\Kp E_i^{1/2}-\Km E_i^{-1/2})\Vv^I\big]
\nonumber\\&\iff
\partial_t\Vv^I=\delta_i
\big[(\Kp E_i^{1/2}-\Km E_i^{-1/2})\Vv^I\big]+D\Vv^I.
\label{eq:macropatchemb}
\end{align}
The operator on the right-hand side is precisely the same as that for the microscale. 
Thus, in this patch scheme, the evolution over the macroscale of the mid-patch values~\(\Vv^I\) are consistent with the microscale evolution~\cref{eqlatvec}.
\end{proof}}

Consequently, any errors in the macroscale of this patch scheme applied to~\cref{eq:mbd1Ddiff} arise only from errors in the interpolation of the edge values (and the usual round-off errors).

\begin{corollary}\label{corlatvec}
The patch scheme~\eqref{eqlatc} (or \cref{sec:ens}) applied to the ensemble~\eqref{eq:mbd1Ddiff} of the heterogeneous diffusion~\eqref{eq:full1Ddiff}, is  consistent with the homogenisation~\eqref{eq:slowman} of the heterogeneous diffusion~\eqref{eq:full1Ddiff}. 
\end{corollary}

\begin{proof}
Recall that~\eqref{eq:mbd1Ddiff} is an ensemble of \(p\)~uncoupled, phase-shifted, copies of the heterogeneous diffusion~\eqref{eq:full1Ddiff}.
So the results for~\eqref{eq:mbd1Ddiff} apply to~\eqref{eq:full1Ddiff}.
Also, the centre manifold theory supported homogenisation~\eqref{eq:slowman} is a `\pde' of a superposition, a linear combination, a low-pass filter, of exact solutions of the ensemble~\eqref{eq:mbd1Ddiff}, and hence of~\eqref{eq:full1Ddiff}. 
\cref{thm:psc} proves the patch scheme~\eqref{eqlatc} is consistent with the ensemble~\eqref{eq:mbd1Ddiff}, and so it is consistent with the homogenisation~\eqref{eq:slowman} of the heterogeneous diffusion~\eqref{eq:full1Ddiff}.
\end{proof}

\begin{corollary}\label{corlatmul}
If the chosen size~\(n\) of the patches is a multiple of the microscale periodicity~\(p\), then the patch scheme applied directly to the heterogeneous diffusion~\cref{eq:full1Ddiff} (without the ensemble of phase shifts) is consistent with the homogenisation of~\cref{eq:full1Ddiff}. 
\end{corollary}
\begin{proof} 
Recall that, among the ensemble of phase-shifted diffusivities, $v^I_{i,\ell}$~is the member with phase $\phi=(i-\ell \mod p)$.
Further recall that coupling is always between edge fields and next-to-edge fields with the same~$\ell$ value, so that $v^I_{0,\ell}$ is computed from interpolations of $v^J_{n,\ell}$, and $v^I_{n+1,\ell}$ is computed from interpolations of $v^J_{1,\ell}$, for some patches~$J$ neighbouring patch~$I$.
For most~\(n\) the inter-patch coupling~\eqref{eqlatc} couples patches via different phases.
For example, \cref{fig:1Dcoupling43,fig:embedn3} with $n=4$ and periodicity $p=3$ couples the right-edge field of phase $\phi=0$ ($i=5$\,, $\ell=2$) to the left-next-to-edge fields of $\phi=2$ ($i=1$\,, $\ell=2$) in neighbouring patches, and the left-edge field of phase $\phi=0$ ($i=0$\,, $\ell=0$) to the right-next-to-edge fields of $\phi=1$ ($i=4$\,, $\ell=0$) in neighbouring patches.
In general, the inter-patch coupling~\cref{eqlatc} for $\ell=0:p-1$ has phase~$\phi$ on the right-edge  of patch~\(I\) coupled with phase~\(\phi-n\mod p\) left-next-to-edge fields in neighbouring patches, and on the left-edge of patch~\(I\) coupled with phase~\(\phi+n\mod p\) in  right-next-to-edge fields neighbouring patches.
Consequently, in the cases when $p$~divides~$n$ the inter-patch coupling~\cref{eqlatc} always couples with the same phase from every patch.
In these cases the $p$~phases do not interact with each other. 
That is, when $p$~divides~$n$ the patch scheme of an ensemble decouples to $p$~independent equivalent patch systems, and so we need only compute one of these $p$~systems to obtain the benefits of the ensemble results.
Thus, \cref{corlatvec} assures us that when $p$~divides~$n$ the patch scheme applied to the heterogeneous diffusion~\cref{eq:full1Ddiff} is consistent with the homogenisation~\eqref{eq:slowman} of~\cref{eq:full1Ddiff}.
\end{proof}

\section{Self-adjoint preserving patch scheme for 2D}
\label{sec:pd2D}

Recall that \cref{fig:homoEG2} illustrates a patch scheme simulation that provides a computational homogenisation of the heterogeneous 2D diffusion~\cref{eq:full2Ddiff} on a microscale lattice.
We explore the patch scheme applied to~\cref{eq:full2Ddiff} as it is the canonical example of the homogenisation of heterogeneous \pde{}s in multiple spatial dimensions.
We first define how to coupling patches in 2D, confirm that the 2D patch scheme is self-adjoint, and then verify the accuracy when using spectral coupling (\cref{sec:sa2D}).
For 2D heterogeneous diffusion~\cref{eq:full2Ddiff}, patch coupling is constrained in the same way as the 1D heterogeneous diffusion patch coupling; namely that we require the patch size to be divisible by the period of the heterogeneous diffusion; but like the 1D case, an ensemble of phase-shifts relaxes this constraint (\cref{sec:ens2D}).
Finally, we prove that the resultant 2D patch scheme is consistent with the macroscale dynamics of the original microscale system (\cref{secpsc2D}).
We conjecture there is a straightforward extension of this patch scheme and theory from 2D to higher dimensions.

\subsection{Self-adjoint coupling for 2D}
\label{sec:sa2D}

Consider the 2D heterogeneous diffusion~\eqref{eq:full2Ddiff}. 
For this 2D~case, we create a patch scheme using similar parameters to those defined for the 1D case (e.g.,~$d$, $H$, $N$, $n$,~$r$), but now with subscripts $x$~and~$y$ to distinguish the parameters for the $x$~and~$y$ directions (\cref{fig:2Dcoupling}).
The patch scheme divides the 2D~domain of this microscale model into well-separated patches on a macroscale lattice indexed by~$(I,J)$ with macroscale lattice spacing $H_x$~and~$H_y$, as illustrated by \cref{fig:homoEG2,fig:2Dcoupling}.
In the patch scheme, the microscale indices~$(i,j)$ index interior points within a patch by $i=1:n_x$ and $j=1:n_y$\,. 
On all such interior points we apply the 2D heterogeneous diffusion~\eqref{eq:full2Ddiff} but with the patch index specified:
\begin{subequations}\label{eqs:2Ddiffpatch}%
\begin{align}
\partial_t u^{I,J}_{i,j}&=\big[\kappa^{I,J}_{i+\frac12,j}(u^{I,J}_{i+1,j}-u^{I,J}_{i,j})
+\kappa^{I,J}_{i-\frac12,j}(u^{I,J}_{i-1,j}-u^{I,J}_{i,j})
\big]/d_x^2
\nonumber\\&\quad{}
+\big[\kappa^{I,J}_{i,j+\frac12}(u^{I,J}_{i,j+1}-u^{I,J}_{i,j})
+\kappa^{I,J}_{i,j-\frac12}(u^{I,J}_{i,j-1}-u^{I,J}_{i,j})
\big]/d_y^2\,,
\label{eq:2Ddiffpatch}
\end{align}
where the diffusivities have lattice periods~$p_x$ and~$p_y$ in the $x$~and~$y$ directions, respectively.

\begin{figure}
\centering
\caption{\label{fig:2Dcoupling}self-adjoint preserving,  nearest neighbour ($\Gamma=1$), coupling of patches of size \(5\times5\)\,:
(left)~vertical~$y$  coupling; and (right)~horizontal~$x$ coupling.
Filled {\color{blue} triangles}\slash {\color{red} circles} are the next-to-edge points whose values are interpolated to the edge points at unfilled {\color{blue} triangles}\slash {\color{red} circles}.
Interpolation to a {\color{blue} bottom}\slash {\color{red} top} edge is from the next-to-edge points with the same horizontal position on the {\color{blue} top}\slash {\color{red} bottom} of the near patches above and below, as indicated by arrows in the left diagram.
Similarly for {\color{blue} left}\slash {\color{red} right} interpolation.
}
\includegraphics{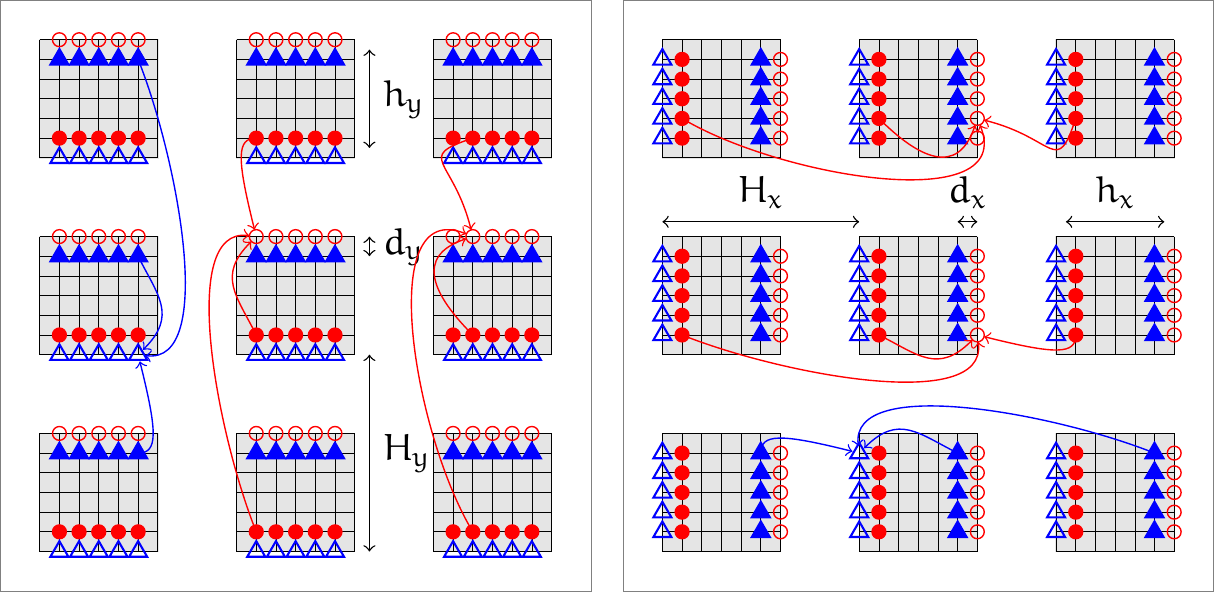}
\end{figure}

The diffusion equation~\eqref{eq:2Ddiffpatch} requires field values at the edges of every patch, namely $u^{I,J}_{0,j}$, $u^{I,J}_{n_x+1,j}$, $u^{I,J}_{i,0}$ and $u^{I,J}_{i,n_y+1}$ for $i=1:n_x$ and $j=1:n_y$\,.
The inter-patch coupling specifies these edge values. 
\cref{fig:2Dcoupling} illustrates the 2D~patch coupling across both~$x$ (to obtain left and right patch-edge values) and~$y$ (to obtain bottom and top patch-edge values).
We implement the following 2D~form of the 1D coupling~\cref{eq:sacoup} in the \(x\)~and~\(y\) directions, respectively:
\begin{align}&
u^{I,J}_{0,j}=\sum_K\cI^{IK}_{1n_x}u^{K,J}_{n_x,j}\quad \text{and} \quad u^{I,J}_{n_y+1,j}=\sum_K\cI^{IK}_{n_x1}u^{K,J}_{1,j}\,,\quad  j=1:n_y\,;
\label{eq:2Dcouplex}\\&
u^{I,J}_{i,0}=\sum_K\cJ^{JK}_{1n_y}u^{I,K}_{i,n_y}\quad \text{and} \quad u^{I,J}_{i,n_y+1}=\sum_K\cJ^{JK}_{n_y1}u^{I,K}_{i,1}\,,\quad i=1:n_x\,.
\label{eq:2Dcoupley}
\end{align}
\end{subequations}
The matrices $\cI$~and~$\cJ$ for $x$~and~$y$ coupling, respectively, are the diffusivity independent matrices of interpolation coefficients and are equivalent to~$\cI$ for 1D~systems defined in \cref{sec:sa1D}, but  depend on different size ratios $r_x=h_x/H_x$ and $r_y=h_y/H_y$\,, respectively.
As in 1D, we choose these coefficients to implement either spectral coupling (\cref{sec:sasc}), or Lagrangian coupling (\cref{sec:fwsc}).

\cref{fig:homoEG2} illustrates a patch simulation of the heterogeneous diffusion~\eqref{eqs:2Ddiffpatch} for a Gaussian initial condition on a macroscale lattice of $5\times 5$ patches.
The patches are of size $n_x=n_y=3$ in both directions.
The \(18\)~microscale diffusion coefficients~$\kappa_{i,j}$ have a period of three in both directions and are independently and identically distributed log-normally (proportional to~\(\exp[2\cN(0,1)]\)). 
The patches are coupled via two sets of 1D~spectral interpolation (\cref{sec:sasc}).
From the initially smooth Gaussian, the patch system evolves to a `rough' sub-patch structure that reflects the microscale heterogeneity.
Thereafter, the patch system simulates how the field~\(u\) diffuses across the domain according to an effective macroscale anisotropic homogenisation.

\cref{fig:2Dsin} simulates a similar heterogeneous diffusion but with rectangular patches so that there are geometric differences between the two principal directions.  
Here the microscale diffusivities are (to two decimal places)
{\small
\begin{align}\label{eq:chetr}
\begin{bmatrix} \kappa_{i+\frac12,j} \end{bmatrix}
=\begin{bmatrix}
18.91  &  1.06  &  0.63 &   2.11\\
    4.46  &  0.72  &  1.02  &  1.66\\
    4.89  &  0.88  &  1.31  &  5.79\\
    1.62  &  2.68  &  2.32  &  1.24\\
    0.42  &  0.88 &   0.59  &  1.35
    \end{bmatrix},
&&
\begin{bmatrix} \kappa_{i,j+\frac12} \end{bmatrix}
=\begin{bmatrix}
    0.48  &  0.63  &  1.31  &  0.51\\
    0.39  & 10.38  &  3.07 &   0.37\\
    2.10  &  1.74  &  2.68  &  1.63\\
    1.20  &  4.38  &  0.50  &  1.02\\
    2.55  &  1.23  &  0.33  &  1.06
    \end{bmatrix},
\end{align}
}%
for microscale periodicities \(p_x=n_x=5\) and \(p_y=n_y=4\), with patch width ratios \(r_x=0.5\) and \(r_y=0.4\)\,.
The qualitative properties of the patch simulation are like those of \cref{fig:homoEG2} albeit here starting from a noisy initial condition.
Such simulations as these are readily performed using our \textsc{Matlab}\slash Octave Toolbox \cite[]{Maclean2020a}.

\begin{figure}
\caption{\label{fig:2Dsin} Simulation of 2D heterogeneous diffusion~\cref{eqs:2Ddiffpatch} with diffusivities~\cref{eq:chetr}, spectral coupling, and a noisy initial condition.
}
\pgfplotsset{label shift=-2ex}
\begin{tabular}{@{\hspace{-2mm}}c@{}c@{}}
\includegraphics{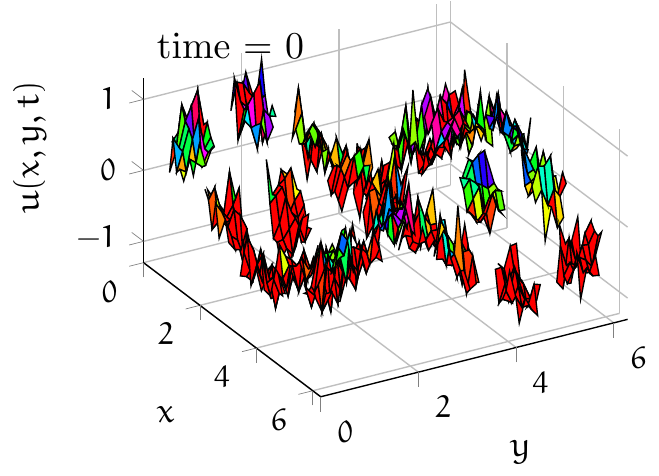}&
\includegraphics{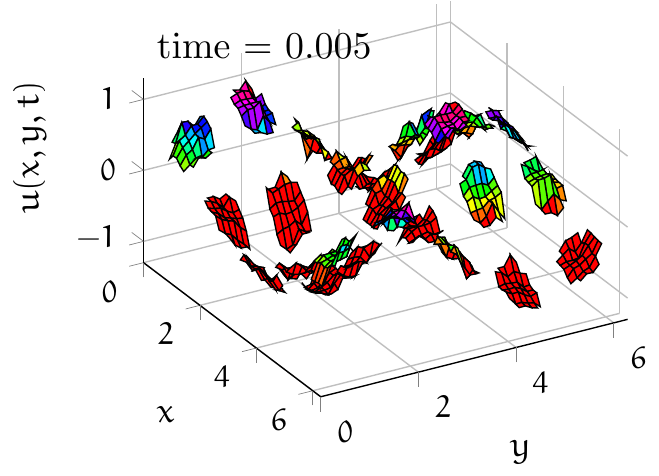}\\
\includegraphics{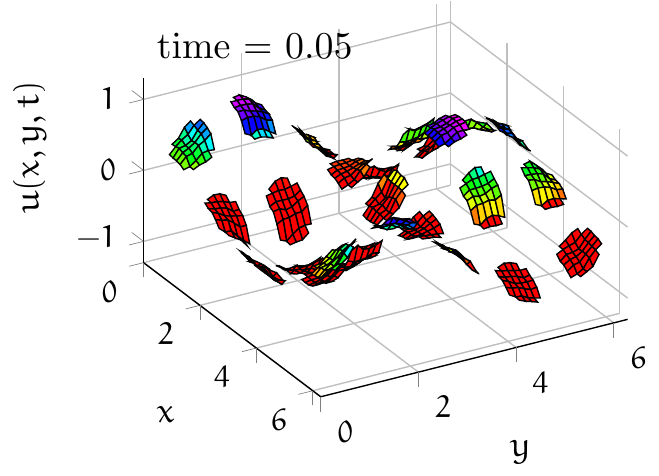}&
\includegraphics{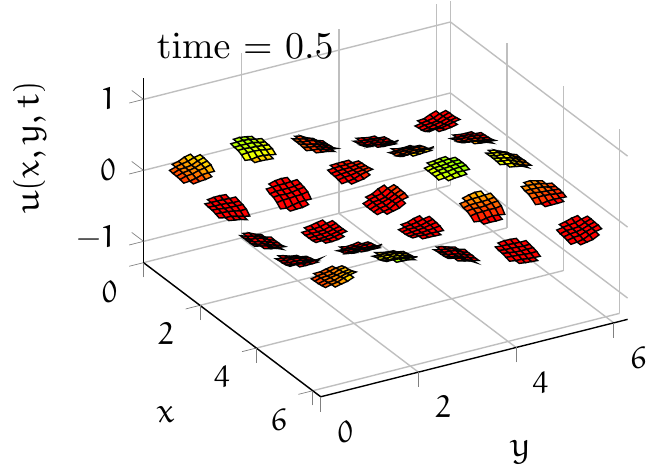}
\end{tabular}
\end{figure}

\begin{corollary}[2D self adjoint]\label{cor2Dsa}
The 2D patch scheme~\eqref{eqs:2Ddiffpatch} preserves the self-adjoint symmetry of the 2D heterogeneous diffusion~\eqref{eq:full2Ddiff} when the patches are coupled by spectral (\cref{sec:sasc}) or Lagrangian (\cref{sec:fwsc}) interpolation.
\end{corollary}

The following proof should generalise by induction to also cover corresponding patch schemes in multi-D space.

\begin{proof}
Let the space \(\HH:=\RR^{n_yN_y}\), and form all field values at fixed~\(x\), and all~\(y\), into the vectors \(v^I_i:=\big( u^{I,J}_{i,j}\C j=1:n_y\C J=1:N_y\big)\in\HH\)\,.
Correspondingly form the diffusivity matrices that operate in the \(x\)-direction as \(\cK^I_{i\pm1/2}:=\diag\big(\kappa^{I,J}_{i\pm1/2,j}\C j=1:n_y\C J=1:N_y\big)\in\HH\times\HH\)\,. 
Then the 2D heterogeneous diffusion~\eqref{eq:2Ddiffpatch} becomes
\begin{equation}
\partial_tv^I_i=\big[\cK^I_{i+\frac12}(v^I_{i+1}-v^I_i)
+\cK^I_{i-\frac12}(v^I_{i-1}-v^I_i)
\big]/d_x^2+\cL^I_iv^I_i\,,
\label{eq2Dsa}
\end{equation}
where \(\cL^I_i\in\HH\times\HH\) is the matrix of \(y\)-direction  interactions and coupling~\eqref{eq:2Dcoupley} at fixed~\(x\).
By \cref{lemsasp,lemsafw}, the~\(\cL^I_i\) are self-adjoint, and so the system-wide \(\diag\big(\cL^{I}_{i}\C i=1:n_x\C I=1:N_x\big)\) is self-adjoint.
The term \([\cdots]/d_x^2\) in~\eqref{eq2Dsa},  coupled by~\eqref{eq:2Dcouplex}, is in a generalised form of the 1D diffusion~\eqref{eq:1Ddiff}, and so, by \cref{remHilb,lemsasp,lemsafw}, it also is self-adjoint.
Thus \eqref{eq2Dsa} coupled by~\eqref{eq:2Dcouplex} is self-adjoint, and hence so is the patch scheme~\eqref{eqs:2Ddiffpatch}.
\end{proof}

\paragraph{Numerics verify accuracy of the scheme}
To verify the accuracy of the patch scheme we compare the full-space lattice dynamics of heterogeneous diffusion~\cref{eq:full2Ddiff} with the two cases of the corresponding patch scheme~\cref{eqs:2Ddiffpatch}, for two size ratios, \(r_x,r_y=0.5\) and~\(0.25\).
The underlying microscale lattice system is kept the same across these three cases (the same~\(d_x,d_y\) and diffusivities).
We tried dozens of cases where the microscale periodicities divide the patch sizes, and all verified that the 2D~patch scheme with spectral interpolation, to computer round-off error, \begin{itemize}
\item preserved the self-adjoint symmetry, and 
\item correctly predicted the (small magnitude) macroscale eigenvalues.
\end{itemize}
That is, the spectral patch scheme appears to be an accurate computational homogenisation (as proved subsequently by \cref{corslat2c}).

\begin{figure}
\caption{\label{fig:2Derr} Relative errors of macroscale eigenvalues from Lagrangian coupling compared to spectral coupling for 2D heterogeneous diffusion~\cref{eqs:2Ddiffpatch} with diffusivities~\cref{eq:chetr} of periodicity $p_x=5=n_x$\,, $p_y=4=n_y$\,.
Other patch parameters are  $N_x=10$, $N_y=11$, $r_x=0.5$\,, $r_y=0.4$\,.
The right-hand column lists the eigenvalues from spectral coupling.
}
\centering
\includegraphics{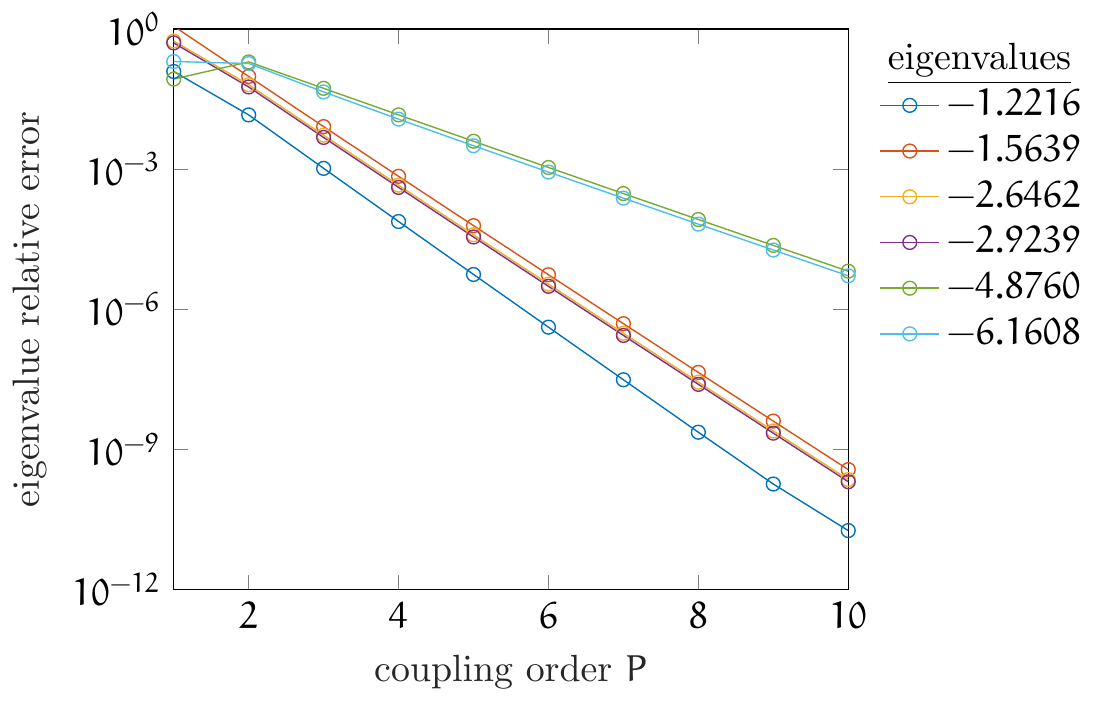}
\end{figure}

As an example of Lagrangian coupling, let's continue considering the diffusivities of~\eqref{eq:chetr} with the same parameters.
\cref{fig:2Derr} plots relative errors of the macroscale eigenvalues of the 2D heterogeneous diffusion~\cref{eqs:2Ddiffpatch} for Lagrangian coupling of various orders~\(\Gamma\).
The errors are relative to the accurate spectral coupling. 
As in the 1D case (\cref{fig:homoEigen1}), \cref{fig:2Derr} shows the expected exponential decrease in macroscale errors as the coupling order~$\Gamma$ increases.

\subsection{An ensemble of phase-shifts appears accurate}
\label{sec:ens2D}

As in \cref{sec:ens} for~1D, simulating over an ensemble of phase-shifts provides flexibility as then there are no constraints on the size of the patches.
Further, \cref{corlatvec} proves that the 1D~patch scheme applied to the ensemble is consistent to arbitrarily high-order with the homogenisation of the heterogeneous diffusion.
Here we discuss how to extend the 1D~ensemble to the 2D~case.

In 2D, the ensemble of all phase shifts of the diffusivities, such as~\eqref{eq:chetr}, is constructed from $p_x$~phase-shifts in the $x$~directions combined with $p_y$~phase-shifts in the $y$~direction, to give a total of $p:=p_xp_y$ members in the ensemble.
For computation we let \(\uv^{I,J}_{i,j}(t)=(u^{I,J}_{i,j,0}(t),\ldots,u^{I,J}_{i,j,p-1}(t))\in\RR^p\) be the vector over the ensemble of field values at position~\((x^I_i,y^J_j)\).

The ensemble microscale system within each patch is the \ode{}s~\cref{eq:2Ddiffpatch} for every member of the ensemble.
For the ensemble as a whole~\cref{eq:2Ddiffpatch} becomes the system
\begin{subequations}\label{eq:cens2D}%
\begin{align}
\partial_t \uv^{I,J}_{i,j}&=\big[\cK^{I,J}_{i+\frac12,j}(\uv^{I,J}_{i+1,j}-\uv^{I,J}_{i,j})
+\cK^{I,J}_{i-\frac12,j}(\uv^{I,J}_{i-1,j}-\uv^{I,J}_{i,j})
\big]/d_x^2
\nonumber\\&\quad{}
+\big[\cK^{I,J}_{i,j+\frac12}(\uv^{I,J}_{i,j+1}-\uv^{I,J}_{i,j})
+\cK^{I,J}_{i,j-\frac12}(\uv^{I,J}_{i,j-1}-\uv^{I,J}_{i,j})
\big]/d_y^2\,,
\label{eq:2DdiffpatchE}
\end{align}
where diffusivity matrices \(\cK_{i,j}:=\diag\big(\kappa^{I,J}_{i,j},\text{ in ensemble order}\big)\).

Then patches of the ensemble~\eqref{eq:2DdiffpatchE} are coupled by a variant of~\eqref{eq:2Dcouplex,eq:2Dcoupley}.
Let's introduce permutation matrices that encode the analogue of the `tangle' of inter-patch communication illustrated by \cref{fig:1Dcoupling43}.
Choose the order of the ensemble in each patch so that the diffusivities~\(\kappa^{I,J}_{i,j}\) are independent of~\(I,J\), which can always be done as all phase-shifts are in the ensemble.
Denote the ensemble vector of diffusivities, in order, on the left patch-edge to be \(\kappav_l:=(\kappa_{1/2,j})\in\RR^p\),  denote those on the right patch-edge by \(\kappav_r:=(\kappa_{n_x+1/2,j})\in\RR^p\), and then set the permutation matrix~\(\cP_x\) such that for all~\(\kappa_{i,j}\) the identity \(\kappav_l=\cP_x\kappav_r\) holds.  
This permutation matrix then connects the right-edge values to the left-edge of the appropriate member of the ensemble.  
Similarly, set the permutation matrix~\(\cP_y\) to connect the top-edge values to the bottom-edge of the appropriate member of the ensemble.
Then the inter-patch coupling conditions for the ensemble are
\begin{align}&
\uv^{I,J}_{0,j}=\cP_x\sum_K\cI^{IK}_{1n_x}\uv^{K,J}_{n_x,j}\,, \quad \uv^{I,J}_{n_y+1,j}=\cP_x^\dag\sum_K\cI^{IK}_{n_x1}\uv^{K,J}_{1,j}\,,\quad  j=1:n_y\,;
\label{eq:cens2Dx}\\&
\uv^{I,J}_{i,0}=\cP_y\sum_K\cJ^{JK}_{1n_y}\uv^{I,K}_{i,n_y}\,, \quad \uv^{I,J}_{i,n_y+1}=\cP_y^\dag\sum_K\cJ^{JK}_{n_y1}\uv^{I,K}_{i,1}\,,\quad i=1:n_x\,,
\label{eq:cens2Dy}
\end{align}
\end{subequations}
where coefficients~\(\cI^{IJ}_{ij}\) and~\(\cJ^{IJ}_{ij}\) implement, in the \(x\) and~\(y\) directions respectively, either spectral interpolation (\cref{sec:sasc}) or Lagrangian interpolation (\cref{sec:fwsc}).
Naturally generalising \cref{cor2Dsa} to vector systems then asserts that this patch-ensemble system preserves the self-adjoint symmetry of the microscale heterogeneous diffusion.

Extensions to more space dimensions appear to be straightforward.

\begin{figure}
\caption{\label{fig:2Derrens} Relative errors of macroscale eigenvalues from Lagrangian coupling for the 2D heterogeneous diffusion~\cref{eq:2Ddiffpatch}, with periodicity $p_x=5$\,, $p_y=4$\,.
We implement the ensemble of all twenty phase-shifts of the diffusivities~\cref{eq:chetr}. 
Patch parameters are $N_x=10$, $N_y=11$\,, width ratios $r_x,r_y=0.1$\,, and with only $n_x=n_y=1$\,. 
This figure and \cref{fig:2Derr} have the same microscale diffusion, the same~$d_x$ and~$d_y$.
The right-hand column lists the accurate eigenvalues obtained \text{from spectral coupling.}
}
\centering
\includegraphics{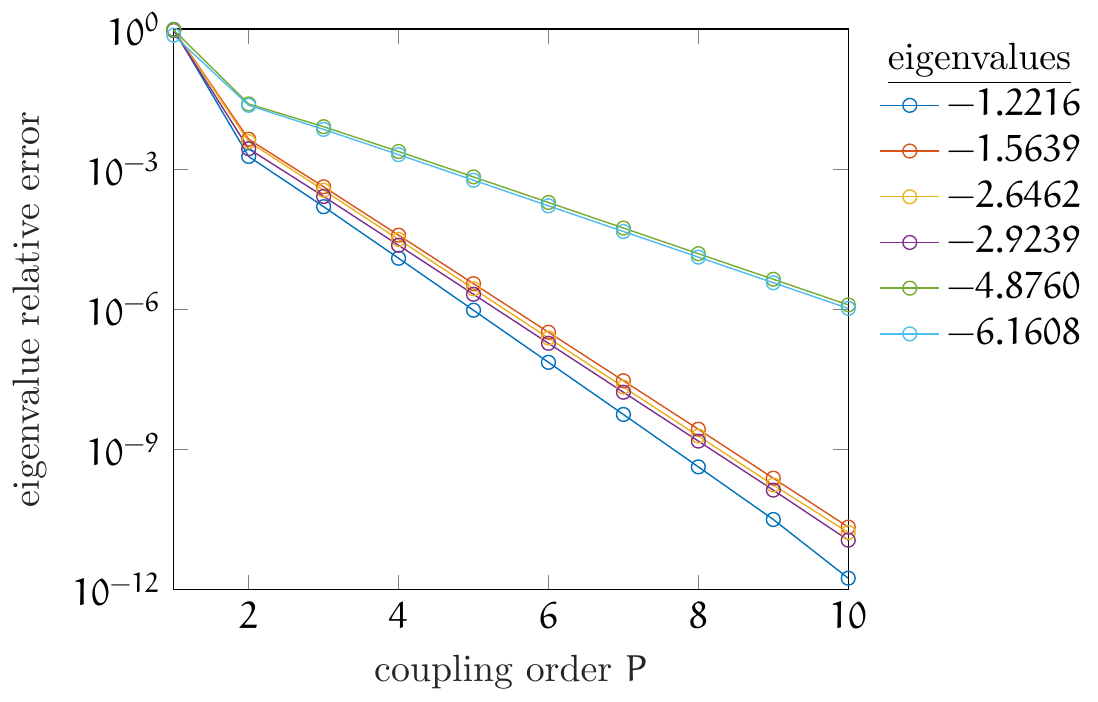}
\end{figure}

\paragraph{Verify accuracy in an example}
Consider the example of this patch-ensemble scheme~\eqref{eq:cens2D}, with Lagrangian coupling of order~\(\Gamma\), applied to the heterogeneous diffusion with diffusivities~\cref{eq:chetr}.
The ensemble has $p=p_xp_y=20$ members from all phase-shifts of these diffusivities.
Upon numerically computing the Jacobian of this scheme, 
\cref{fig:2Derrens} plots relative errors of its macroscale eigenvalues, relative to those for spectral coupling. 
As before, we observe the expected exponential decrease in errors with increasing order of coupling~\(\Gamma\).
In order to compare with~\cref{fig:2Derr} with its microscale spacings $d_x\approx 0.063$ and $d_y\approx 0.057$\,, we here decreased the width ratios, and used the smallest possible patch sizes $n_x,n_y=1$\,.
\cref{fig:2Derr,fig:2Derrens} are remarkably similar, although the ensemble here provides a smoother plot and \text{somewhat smaller errors.}

We explored many other parameter choices and observed that the ensemble always provides error plots similar to \cref{fig:2Derrens}, but that the single phase error plots were more variable, with errors in the Lagrangian coupling sometimes adversely affected by a single large diffusion coefficient. 

\cref{fig:2Dsinens} plots a simulation of the ensemble-mean \(\bar u(x^I_i,y^J_j,t):=\frac1p\sum_{\ell=0}^{p-1}u^{I,J}_{i,j,\ell}(t)\) of the patch-ensemble scheme~\eqref{eq:cens2D} for diffusivities~\eqref{eq:chetr}---the 2D heterogeneous diffusion~\cref{eq:full2Ddiff} on patches of the twenty member ensemble of all phase-shifts of the diffusivities.
It used spectral coupling via the  \textsc{Matlab}/Octave Toolbox \cite[]{Maclean2020a}.
For better visualisation we use only $5\times 5$ patches, with other parameters the same as~\cref{fig:2Derr}; that is,  $n_x=5$\,, $n_y=4$ and $r_x=0.5$\,, $r_y=0.4$\,.
The initial condition is sinusoidal in $x$~and~$y$, plus microscale noise. 
The microscale noise dissipates rapidly so that by time~$0.05$ the simulation's ensemble average is smooth (\cref{fig:2Dsinens}).
In contrast, at time~$0.05$ the single phase simulation (\cref{fig:2Dsin}) is `rough' on the microscale, but this roughness is due to the heterogeneous diffusion rather than the initial disorder.

\begin{figure}
\caption{\label{fig:2Dsinens} Simulation of 2D heterogeneous diffusion~\cref{eqs:2Ddiffpatch} with an ensemble of diffusion configurations constructed from~\cref{eq:chetr} and spectral coupling.
}
\pgfplotsset{label shift=-2ex}
\begin{tabular}{@{\hspace{-2mm}}c@{}c@{}}
\includegraphics{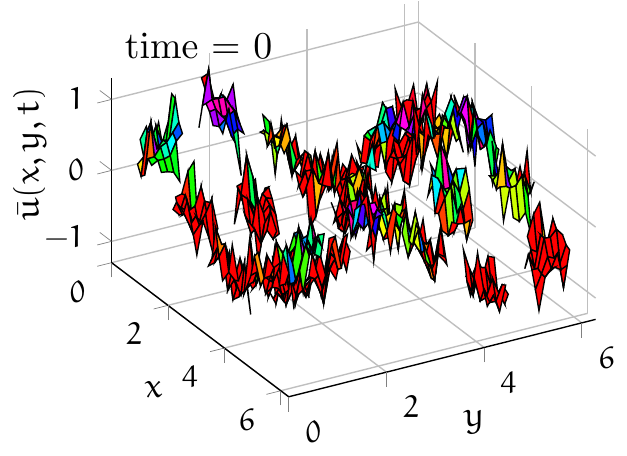}&
\includegraphics{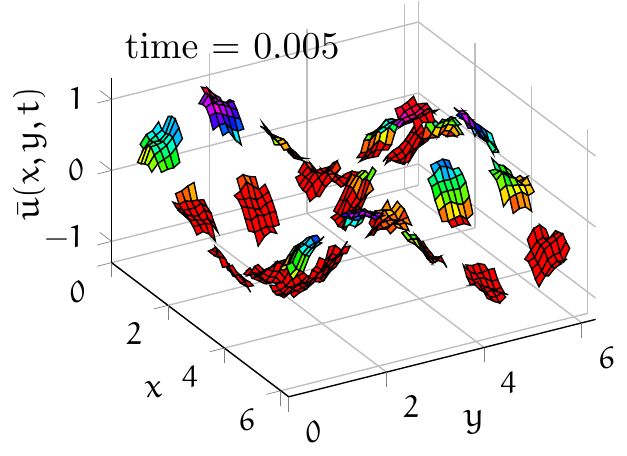}\\
\includegraphics{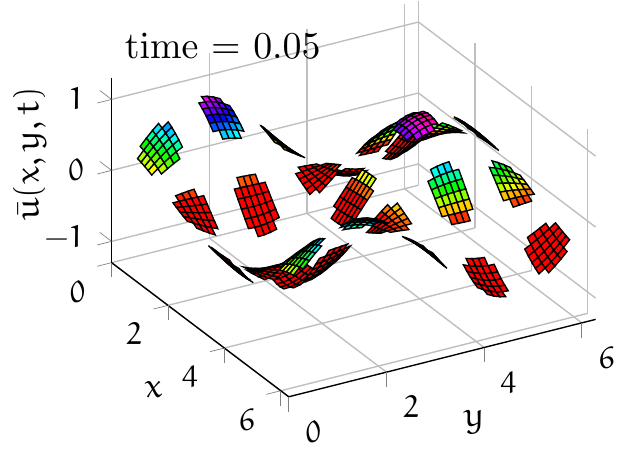}&
\includegraphics{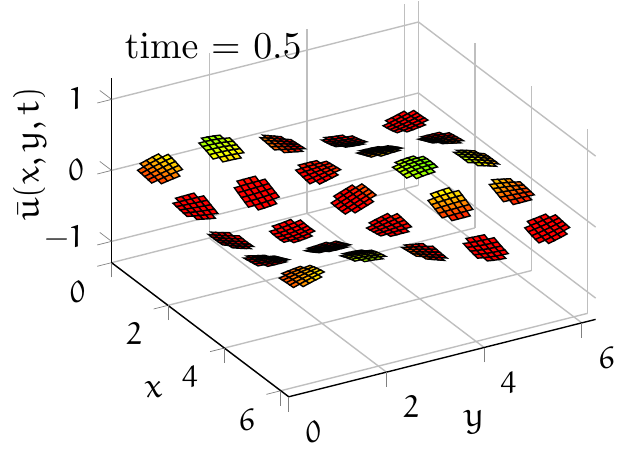}
\end{tabular}
\end{figure}

\subsection{This 2D patch scheme is consistent to high-order}
\label{secpsc2D}
In the previous  \cref{sec:ens2D} our 2D patch dynamics scheme is defined in terms of an ensemble~\cref{eq:cens2D} with heterogeneity, defined by~$\cK$, varying with the spatial parameters.
But for theoretical purposes the ensemble is better formed as a homogeneous system, as in \cref{sec:theory} for~1D where the system~\cref{eq:mbd1Ddiff} is homogeneous in the sense that the diffusion coefficient is independent of the spatial index.

\renewcommand{\vi}[1][i,j]{\vv^{I,J}_{#1}} %

We now form a 2D homogeneous system ensemble analogous to that of the 1D system of \cref{figembedshifts}.
At each point~\((x_i,y_j)\) on the microscale lattice in space, there are \(p=p_xp_y\)~variables forming the ensemble: form them into vector~\(\vv_{i,j}(t)\in\RR^p\). 
Let the components of this vector be denoted~\(v_{i,j,k,\ell}(t)\) for \(k=0:n_x-1\)\,, \(\ell=0:n_y-1\)\,, and in a suitable order for the ensemble.
Rewrite the ensemble equations in the form
\begin{align}
d^2\partial_t v_{i,j,k,\ell }
&=\kappa_{k+\frac12,\ell }(v_{i+1,j,k+1,\ell }-v_{i,j,k,\ell })
+\kappa_{k-\frac12,\ell }(v_{i-1,j,k-1,\ell }-v_{i,j,k,\ell })
\nonumber\\&\quad{}
+\kappa_{k,\ell +\frac12}(v_{i,j+1,i,\ell +1}-v_{i,j,k,\ell })
+\kappa_{k,\ell -\frac12}(v_{i,j-1,i,\ell -1}-v_{i,j,k,\ell })
\label{eq:mbd2Ddiff}
\end{align}
because then the heterogeneity only varies in~\(k,\ell \) (like index~$\ell$ in \cref{figembedshifts} for the 1D case).
Then \(u_{i,j}(t):=v_{i,j,i+\phi,j+\psi}\) satisfies the original heterogeneous diffusion~\eqref{eq:full2Ddiff}, but with the diffusivities phase-shifted by~\(\phi,\psi\), respectively.
That is, the system~\eqref{eq:mbd2Ddiff} captures all the \(p=p_xp_y\) phase-shifted versions of the heterogeneous diffusion~\eqref{eq:full2Ddiff}, but is homogeneous in the spatial indices~$i,j$ and analogous to the homogeneous 1D ensemble system~\cref{eq:mbd1Ddiff}.

Now form the vectors~\(\vv_{i,j}\) from a suitable ordering of~\(v_{i,j,k,\ell}\) so that we can write~\eqref{eq:mbd2Ddiff} in the matrix-vector form
\begin{equation}
\partial_t\vv_{i,j}=
\delta_i\big(\Kp_xE_i^{1/2}-\Km_xE_i^{-1/2}\big)\vv_{i,j}
+\delta_j\big(\Kp_yE_j^{1/2}-\Km_yE_j^{-1/2}\big)\vv_{i,j}
+D\vv_{i,j}\,,
\label{eq:full2DdiffE}
\end{equation}
in terms of microscale shifts~\(E_i\) and~\(E_j\), microscale centred differences~\(\delta_i\) and~\(\delta_j\), some  \(p\times p\) cross-ensemble diffusivity matrix~\(D\), and four \(p\times p\) off-diagonal diffusivity matrices~\(\Kp_x\C \Kp_y\C \Km_x\) and~\(\Km_y\) that are \emph{independent of location}~\(i,j\).
The homogeneous matrix-vector system~\eqref{eq:full2DdiffE} is the 2D analogue of the 1D system~\eqref{eqlatvec}, and is just \(p\)~phase-shifted copies of the original heterogeneous diffusion~\eqref{eq:full2Ddiff}.

To proceed to analyse the patch scheme for 2D heterogeneous diffusion, we apply the patch scheme to the homogeneous~\eqref{eq:full2DdiffE}. 
So, let \(\vi\in\RR^p\)~denote the vector of field values at micro-grid point~\((i,j)\) in patch~\((I,J)\).
In the \(x,y\)-directions, respectively, let there be \(N_x,N_y\) patches, with spacing~\(H_x,H_y\), and of width~\(h_x,h_y\).
Generalising~\eqref{eqlatc} to~2D, apply~\eqref{eq:full2DdiffE} inside the patches:
\begin{subequations}\label{eqlatc2}%
\begin{equation}
\partial_t\vi=
\delta_i\big(\Kp_xE_i^{1/2}-\Km_xE_i^{-1/2}\big)\vi[i,j]
+\delta_j\big(\Kp_yE_j^{1/2}-\Km_yE_j^{-1/2}\big)\vi[i,j]
+D\vi\,,
\label{eqlatvec2}
\end{equation}
with the inter-patch coupling that the edge-values, for every~\(i,j,I,J\), 
\begin{align}\text{(\cref{fig:2Dcoupling} right)\hspace{1em}}&
\vi[n+1,j]=E_x^{r_x}\vi[1,j]\,, \quad
\vi[0,j]=E_x^{-r_x}\vi[n,j]\,, 
\label{eqlatc2x}
\\\text{(\cref{fig:2Dcoupling} left)\hspace{1.6em}}&
\vi[i,n+1]=E_y^{r_y}\vi[i,1]\,, \quad
\vi[i,0]=E_y^{-r_y}\vi[i,n]\,.
\label{eqlatc2y}
\end{align}
\end{subequations}
As in \cref{sec:theory}, the operators~\(E^{\pm r_x}_x\) and~\(E^{\pm r_y}_y\) are shifts across the patch widths, and may be realised and approximated by spectral or Lagrangian interpolation.
Then the following generalises \cref{thm:psc} to~2D.

\begin{theorem} \label{thm:psc2}
The macroscale of the patch scheme~\eqref{eqlatc2} 
is consistent with the microscale dynamics of~\eqref{eq:full2DdiffE} over the entire spatial domain.
\end{theorem}

As in \cref{sec:con}, the subtlety here is that~\eqref{eqlatc2} holds only inside (small) patches of space, whereas~\eqref{eq:full2DdiffE} holds throughout space, thus forming two very different dynamical systems.
The nontrivial challenge is to connect these two systems.

\begin{proof} 
In essence, \cref{thm:psc2} follows because the proof of \cref{thm:psc} applies independently to both the \(x\) and~\(y\) directions.

\begin{itemize}
\item \newcommand{\wi}[1][i]{\wv^{I}_{#1}}
In the first of two steps, consider the \(x\)-direction.
In this step, let the vector \(\wi\) be the vector of values~\((\vi)\in\RR^{pn_yN_y}\) formed over \(j=1:n_y\) and \(J=1:N_y\)\,.  
Then the system~\cref{eqlatc2} takes the form
\begin{equation*}
\partial_t\wi=\delta_i\big(\Kp'\wi[i+1/2]-\Km'\wi[i-1/2]\big)+D'\wi\,,
\end{equation*}
with \(D'\wi\) representing the collective over~\(j,J\) of the \(y\)-direction operator \(\delta_j\big(\Kp_y\vi[i,j+1/2]-\Km_y\vi[i,j-1/2]\big)
+D\vi\) coupled by~\cref{eqlatc2y}, and \(\Kp'\wi[i+1/2]-\Km'\wi[i-1/2]\) representing the collective of \(\Kp_x\vi[i+1/2,j]-\Km_x\vi[i-1/2,j]\).
Further, from~\cref{eqlatc2x}, the \(\wi\)-`patches' are coupled by
\begin{equation*}
\wi[n+1]=E_x^{r_x}\wi[1] \quad\text{and}\quad 
\wi[0]=E_x^{-r_x}\wi[n]\,.
\end{equation*}
This \(\wi\)-system of `patches' is of the form for \cref{thm:psc}, and hence the \(x\)-direction dynamics are consistent.
Specifically, upon defining the mid-`patch' value \(\Wv^I(t):=\wi[\bar n_x]\) for \(\bar n_x:=n_x/2+1/2\)\,, from the last equation in the proof of \cref{thm:psc} we have
\begin{equation*}
\partial_t\Wv^I=\delta_i
\big[(\Kp' E_i^{1/2}-\Km' E_i^{-1/2})\Wv^I\big]+D'\Wv^I.
\end{equation*}
That is, upon unpacking the \(\Wv^I\) variables,
\begin{align}
\partial_t\vi[\bar n_x,j]&=\delta_i
\big[(\Kp_x E_i^{1/2}-\Km_x E_i^{-1/2})\vi[\bar n_x,j]\big]
\nonumber\\&\quad{}
+\delta_j\big(\Kp_y\vi[\bar n_x,j+1/2]-\Km_y\vi[\bar n_x,j-1/2]\big)
+D\vi[\bar n_x,j].
\label{eqlatvec2x}
\end{align}

\item \newcommand{\wj}[1][j]{\wv^{J}_{#1}}
The second of the two steps is to consider the \(y\)-direction.
In this step, let the vector~\(\wj\) be the vector of values~\((\vi[\bar n_x,j])\in\RR^{pN_x}\) formed over \(I=1:N_x\)\,.  
Then the system~\cref{eqlatvec2x} takes the form
\begin{equation*}
\partial_t\wj=\delta_j\big(\Kp''\wj[j+1/2]-\Km''\wj[j-1/2]\big)+D''\wj\,,
\end{equation*}
with \(D''\wj\) representing the collective over~\(I\) of the \(x\)-direction operator \(\big[(\Kp_x E_i^{1/2}-\Km_x E_i^{-1/2})\vi[\bar n_x,j]\big]
+D\vi[\bar n_x,j]\) coupled by~\cref{eqlatc2x}, and \(\Kp''\wj[j+1/2]-\Km''\wj[j-1/2]\) representing the collective of \(\Kp_y\vi[\bar n_x,j+1/2]-\Km_y\vi[\bar n_x,j-1/2]\).
Further, from~\cref{eqlatc2y}, the \(\wj\)-`patches' are coupled by
\begin{equation*}
\wj[n+1]=E_y^{r_y}\wj[1] \quad\text{and}\quad 
\wj[0]=E_y^{-r_y}\wj[n]\,.
\end{equation*}
This \(\wj\)-system of `patches' is of the form for \cref{thm:psc}, and hence the \(y\)-direction dynamics are consistent.
Specifically, upon defining the mid-`patch' value \(\Wv^J(t):=\wj[\bar n_y]=\vi[\bar n_x,\bar n_y]\) for \(\bar n_y:=n_y/2+1/2\)\,, from the last equation in the proof of \cref{thm:psc} we have
\begin{equation*}
\partial_t\Wv^J=\delta_j
\big[(\Kp'' E_j^{1/2}-\Km'' E_j^{-1/2})\Wv^J\big]+D''\Wv^J.
\end{equation*}
That is, upon unpacking the \(\Wv^J\) variables, and defining the mid-patch value \(\Vv^{I,J}:=\vi[\bar n_x,\bar n_y]\),
\begin{align}
\partial_t\Vv^{I,J}&=\delta_j
\big[(\Kp_y E_j^{1/2}-\Km_y E_j^{-1/2})\Vv^{I,J}\big]
\nonumber\\&\quad{}
+\delta_i\big[(\Kp_x E_i^{1/2}-\Km_x E_i^{-1/2})\Vv^{I,J}\big]
+D\Vv^{I,J}.
\label{eqlatvec2y}
\end{align}

\end{itemize}
The operator on the right-hand side of~\eqref{eqlatvec2y} is precisely the same as that for the microscale~\eqref{eq:full2DdiffE}. 
Thus, in the patch scheme~\eqref{eqlatc2}, the evolution over the macroscale of the mid-patch values~\(\Vv^{I,J}\) are consistent with the entire domain evolution of~\eqref{eq:full2DdiffE}.
\end{proof}

The following \cref{corslat2} extends to 2D the properties of \cref{corlatvec,corlatmul}.
\begin{corollary}\label{corslat2}
Consider the heterogeneous diffusion~\cref{eq:full2Ddiff} with diffusivities \(p_x,p_y\)-periodic in the \(x,y\)-directions respectively.
\begin{enumerate}[label={\ref{corslat2}(\alph*)}
    ,ref={Corollary~\thetheorem(\alph*)}]

\item\label{corslat2b} The patch scheme applied to the ensemble~\cref{eq:full2DdiffE} of the heterogeneous diffusion~\cref{eq:full2Ddiff}, is consistent with the homogenisation of the heterogeneous diffusion. 
(This consistency is illustrated by \cref{fig:2Derrens}.)

\item\label{corslat2c} If the chosen sizes of the patches,~\(n_x,n_y\), are a multiple of the microscale periodicities,~\(p_x,p_y\), respectively, then the patch scheme applied to~\cref{eq:full2Ddiff} with only one realisation of the heterogeneity (i.e., without the ensemble) has macroscale dynamics consistent with the homogenisation of the heterogeneous diffusion. 
(This consistency is illustrated by \cref{fig:2Derr}.)
\end{enumerate}
\end{corollary}

{\def\proofname{Proof of \cref{corslat2b}.\quad}
\begin{proof} 
Since~\eqref{eq:full2DdiffE} is formed of \(p\)~decoupled copies of~\eqref{eq:full2Ddiff}, \cref{thm:psc2} implies the patch scheme~\eqref{eqlatc2} of~\eqref{eq:full2DdiffE} is thus also consistent with the original heterogeneous diffusion~\eqref{eq:full2Ddiff}.
\end{proof}}

{\def\proofname{Proof of \cref{corslat2c}.\quad}
\begin{proof} 
In such scenarios each member of the ensemble~\cref{eq:full2DdiffE} decouples from each other.
Thus the patch scheme with only one member of the ensemble, one realisation of the phase-shift in the diffusivities, has the same macroscale dynamics as the patch scheme applied to the ensemble, and hence, by \cref{corslat2b} is consistent with the heterogeneous diffusion~\eqref{eq:full2Ddiff}.
\end{proof}}

\section{Conclusion}

This article solves a longstanding issue in equation-free macroscale modelling by providing accurate and efficient patch coupling conditions which preserve self-adjoint symmetry.
By preserving symmetries we ensure that the macroscale model maintains the same conservation laws as the original microscale model. 
As the self-adjoint symmetry is controlled by the spatial coupling, here we have focused only on the spatial (or gap-tooth~\citep{Gear03}) implementation of patch dynamics, but full implementation with both patch coupling in space and projective integration in time is straight forward with our \textsc{Matlab}/Octave Toolbox \cite[]{Roberts2019b}.

Theoretical support for our self-adjoint macroscale modelling is provided in the context of microscale heterogeneous diffusion and its homogenised dynamics on the macroscale. 
This is a canonical system for multiscale modelling with a rich microscale structure, and provides a basis for macroscale modelling of related but more complex systems. 
For other linear systems of no more than second order in space, accurate macroscale models are expected when implementing  the self-adjoint patch coupling exactly as described here, as shown in the example of a heterogeneous wave (\cref{fig:homoWaveEdgyU2}) and also shown in the \textsc{Matlab}/Octave Toolbox \cite[]{Roberts2019b}.
Further work is required for self-adjoint modelling \text{of nonlinear problems.}

For simulations, this article assumes periodic boundary conditions for the spatial domain. 
Adapting patch dynamics for different spatial boundary conditions is an important future task, and will build upon our current research concerning patch dynamics for shocks~\citep{Maclean2020b}.
The multiscale modelling which accurately captures the sharp features on either side of any shock is expected to be equally effective in capturing boundary layer phenomena. 

In multiscale modelling, an important consideration is how to define suitable macroscale variables which parameterise the slow manifold evolution. 
For the diffusion example considered here, macroscale variables are defined from averages or point samples of microscale variables, but straight forward averaging is often not possible in complex systems.
For example, a microscale description of cell dynamics requires parameters to describe cell locations, velocities and interactions, but the slow macroscale dynamics might be effectively described by only time and cell distribution \citep{Dsilva2018}.
Current research is designing manifold learning algorithms which apply on-the-fly simulations to determine low-dimensional parameterisations on the desired slow manifold, thus efficiently directing the simulation to the manifold of interest \citep{Pozharskiy2020}.
When implemented together, patch dynamics and manifold learning  should provide powerful on-the-fly computation homogenisation.

\paragraph{Acknowledgement}
This research was funded by the Australian Research Council under grants DP150102385 and DP200103097.  
The work of I.G.K. was also partially supported by the DARPA PAI program.

\appendix
\section{Macroscale homogenise of 1D diffusion}
\label{appmh1dd}

This is computer algebra script \verb|homo1Ddiff.txt| for \cref{sssech}.
\begin{lstlisting}
Comment: Rigorous macroscale homogenisation of heterogeneous
diffusion on a lattice, p-periodic.  Embed all phase shifts,
take Fourier transform, and construct slow manifold and
evolution.  Time derivatives scaled by d^2.  AJR, 9/4/2020;

on div; off allfac; on revpri;
p:=3;  % microscale periodicity
matrix z(p+1,1),v(p+1,1),ll(p+1,p+1);
operator ka; % diffusivities ka(j)=kappa_{j-1/2}
let ka(~p)=>p; % optional example values
for j:=1:p do z(j,1):=1; % keep zero in p+1 element
% form matrix of diffusion & solvability
for j:=2:p do ll(j,j-1):=ll(j-1,j):=ka(j);
for j:=1:p-1 do ll(j,j):=-ka(j)-ka(j+1);
ll(p,p):=-ka(1)-ka(p)$ ll(1,p):=ll(p,1):=ll(p,1)+ka(1)$
for j:=1:p do ll(j,p+1):=ll(p+1,j):=-1;
write ll:=ll; % check print
ll:=1/ll$ % form inverse of diffusion & solvability

% macroscale field depends upon time
depend vv,t; 
let df(vv,t)=>v(p+1,1);
v:=vv*z; % initial approximation

% iterate to this order of error
let k^5=>0;
expk:=for n:=0:9 sum (+i*k)^n/factorial(n);
exmk:=for n:=0:9 sum (-i*k)^n/factorial(n);
for it:=1:9 do begin
    % compute residual
    res:=-df(v,t); res(p+1,1):=0;
    res(p,1):=res(p,1)+ka(1)*(expk*v(1,1)-v(p,1))$
    res(1,1):=res(1,1)+ka(1)*(exmk*v(p,1)-v(1,1))$
    for j:=1:p-1 do res(j,1):=res(j,1)+ka(j+1)*(expk*v(j+1,1)-v(j,1));
    for j:=2:p   do res(j,1):=res(j,1)+ka(j  )*(exmk*v(j-1,1)-v(j,1));
    % update field and evolution from residual
    v:=v-ll*res;
    if res=0*z then write "Success ",it:=it+10000;
end;
% write slow manifold and evolution
v:=v;
d2dvvdt:=v(p+1,1);
end;
\end{lstlisting}

\end{document}